\documentclass[12pt]{article} 
\usepackage{a4,amssymb, amsthm, dsfont, float}
\usepackage{amsmath}
\usepackage{mathrsfs,amsfonts}
\usepackage{graphics, epsfig, subfig, xcolor}
\usepackage[mathscr]{euscript} 
\usepackage{bbm, mathtools, nicefrac, float}
\usepackage[normalem]{ulem}
\usepackage{enumerate}
\usepackage[shortlabels]{enumitem}
\usepackage{tikz}
\usetikzlibrary{arrows}
\usetikzlibrary {arrows.meta}
\usepackage{bm}
\usepackage{graphics, epsfig, subfig, xcolor}
\usepackage{algorithm}
\usepackage{algpseudocode}
\usepackage{tabularray}
\DeclareMathOperator*{\esssup}{ess\,sup}
\DeclareMathAlphabet{\mathpzc}{OT1}{pzc}{m}{it}
\newtheorem{thm}{Theorem}[section]

\newtheorem{lem}[thm]{Lemma}
\newtheorem{prop}[thm]{Proposition}
\newtheorem{defn}[thm]{Definition}
\newtheorem{rem}[thm]{Remark}
\newtheorem{assum}[thm]{Assumption}
\numberwithin{equation}{section}

\newtheorem{numresult}[thm]{Numerical Result}
\renewcommand{\theequation}{\arabic{section}.\arabic{equation}}

\begin{document}
\title{Analysis and Numerics  of a Stationary Drift-Diffusion Model for Electrical Discharge in MEMS}

\author{ Runan He 
\thanks{Instituto de Ciencias Matem\'{a}ticas (ICMAT), Calle Nicolas Cabrera 13, 28049 Madrid, Spain.}~ and Wenjia Xie
\thanks{
Feishu, 
No.55 Yueyaquan Road, Putuo District,
200333 Shanghai,
China.}}


\date{\today}

\maketitle
\begin{abstract}

This work presents the analysis and numerical simulation of a stationary drift-diffusion model for electrical discharge in micro-electro-mechanical systems (MEMS). The model couples Poisson’s equation for the electrostatic potential with continuity equations for positive ions and electrons, incorporating a Townsend-type impact ionization source term that depends exponentially on the electric field magnitude.
We prove the existence of weak solutions under physically relevant assumptions and establish uniform bounds on the carrier densities. The proof relies on a regularization–approximation scheme with truncated nonlinearities, monotone operator theory (Browder–Minty), iterative energy estimates, and Stampacchia-type truncation arguments.
We further develop a robust finite element solver to simulate the carrier density and electrostatic potential profiles for several geometries, including two-dimensional domains and a three-dimensional axisymmetric geometry.


\end{abstract}


\section{Introduction}
In this work, we study the existence of weak solutions and the numerical simulation of a stationary variant of the drift-diffusion system that models electrical discharge in micro-electro-mechanical systems (MEMS) and incorporates a Townsend-type ionisation source term. This extension is motivated by the observation that the narrow gas gaps in MEMS devices generate intense local electric fields, under which impact ionisation of gas molecules becomes significant, leading to electron avalanche growth and eventual electrical breakdown \cite{townsend1915electricity,xiao2016gas}. The specific system under investigation reads
\begin{subequations}\label{S_EPSys}
	\begin{equation}
		-\Delta \phi=(p-n)/\epsilon_\phi,\quad x\in\Omega,
	\end{equation}
	\begin{equation}\label{S_parab_eq_pos}
		-\nabla\cdot \left(\epsilon_p\nabla p+p\left(\epsilon_p\nabla \phi-\mathbf{v}\right)\right)=F(n,|\nabla \phi|),\quad x\in\Omega,
	\end{equation}
	\begin{equation}\label{S_parab_eq_neg}
		-\nabla\cdot \left(\epsilon_n\nabla n-n\left(\epsilon_n\nabla \phi+\mathbf{v}\right)\right)=F(n,|\nabla\phi|),\quad x\in\Omega,
	\end{equation}
	\begin{equation}\label{nonlinearity}			F(n,|\nabla\phi|)=\epsilon_nn|\nabla\phi|\left(\alpha_1 e^{-\alpha_2/|\nabla\phi|}-\eta_0\right),
	\end{equation}
\end{subequations}
supplemented by mixed Dirichlet-Neumann boundary conditions:
\begin{subequations}\label{bdyval}
	\begin{equation}\label{Dirichletbdy}
		\phi(x)=\phi_D(x),\quad p(x)=p_D(x),
		\quad n(x)=n_D(x),\quad  
		x\in \partial\Omega_D,  
	\end{equation}
	\begin{equation}\label{Neumannbdy}
		\nu\cdot\nabla \phi=\nu\cdot \nabla p =
		\nu\cdot \nabla n=0,\quad x\in \partial\Omega_N.
\end{equation}\end{subequations}
where $\phi=\phi(x)$, $p=p(x)$ and $n=n(x)$ are unknown functions, representing the electrical potential and the charge densities of positive ions and electrons, respectively. The region $\Omega\subset\mathbb{R}^{d}$ with $d=2,~3$ represents the gas-filled part of the interior of a MEMS device and has boundary $\partial\Omega$ consists of a Dirichlet part $\partial\Omega_D$ and a Neumann part $\partial\Omega_N$: $\partial\Omega = \overline{\partial\Omega_D \cup \partial\Omega_N}$ and $\partial\Omega_D \cap \partial\Omega_N = \emptyset$.  $\nu$ denotes the outward unit normal vector on $\partial\Omega_N$.  The positive constants $\epsilon_\phi, \epsilon_p, \epsilon_n$ are (scaled) permittivity and diffusion coefficients, while $\mathbf{v}:\Omega\to\mathbb{R}^d$ denotes a prescribed velocity of the  incompressible gas flow in a MEMS device, so that $\nabla\cdot\mathbf{v}=0$ in $\Omega$ and tangential to the boundary $\nu\cdot\mathbf{v}=0$ on $\partial\Omega$. The parameters $\alpha_1, \alpha_2 > 0$ characterize the magnitude and activation field of the ionization gate, and $\eta_0 \geq 0$ represents a background attachment rate. 
The Dirichlet data $\phi_D, p_D, n_D$ model Ohmic contacts where thermal equilibrium is assumed, while the homogeneous Neumann conditions represent insulating or symmetry boundaries.

The ionization source term $F(n,|\nabla\phi|)$ is inspired by the classical Townsend theory of electron avalanches in gases \cite{townsend1915electricity,townsend1910theory}. In its original form, the first Townsend coefficient describes the number of ionizing collisions per unit length and is proportional to $e^{-E_0/|\nabla\phi|}$, where $E_0$ is a threshold field strength. In the MEMS context, such field-driven discharge arises in the narrow gaps between electrodes (see, for example, MEMS capacitor, Figure \ref{MEMSsketch}), where the high local field strengths characteristic of microscale geometries can trigger avalanche ionization and breakdown. The smooth exponential gate adopted in \eqref{S_EPSys} avoids the mathematical difficulties associated with hard-cutoff models used in corona discharge studies \cite{budd1991coronas,morrow1997streamer}, while retaining the essential physics of field-dependent ionization with a well-defined activation threshold.

More is said about the physical significance of the problem in the  companion paper \cite{gimperlein2026analysis}.

We shall then prove the following well-posedness result.
\begin{thm}[Existence of Weak Solutions]\label{SS_Existence}
	Under Assumption \ref{assumption} below, there exists a weak solution  $(\phi, p, n)\in (H^{1}(\Omega))^3$ of the problem \eqref{S_EPSys}-\eqref{bdyval} which satisfies $0\leq p,~ n\leq K=\mathrm{const.}<\infty$ almost everywhere in $\Omega$. 
\end{thm}
\begin{numresult} The bounds in Theorem \ref{SS_Existence} are not merely an analytical artefact: they are enforced ``by construction" in our numerical scheme. Writing $p=p_{\min}+e^P$ and $n=n_{\min}+e^N$  for the carrier densities guarantees strict positivity at the discrete level, while the uniform upper bound $K$ justifies the boundedness of the regularised ionisation source. Combined with Streamline Upwind Petrov-Galerkin (SUPG) stabilisation and a Gummel–Newton iteration, this yields a Finite Element Method (FEM) solver that reproduces the predicted behaviour on the geometries of the domains in Section \ref{numerical}, including re-entrant corners and Dirichlet discontinuities where the field is singular. The numerical results thus complement the existence theory, confirming the qualitative structure of the weak solution in regimes beyond the reach of the analysis.
    
\end{numresult}
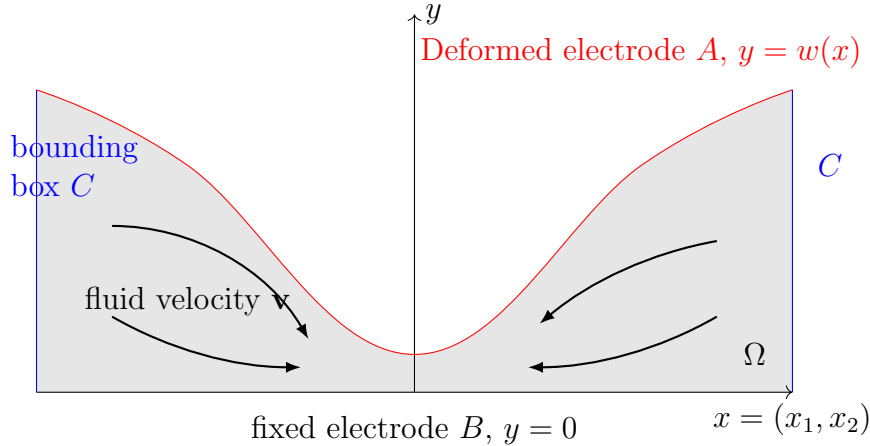
\begin{figure}[h]
	\centering
	\begin{tikzpicture}
		\fill[gray!20, xshift=5cm]
		    plot [smooth, tension=0.7] coordinates { (-5,4) (-3,3) (0,0.5) (3,3) (5,4)}
		    -- (5,0) -- (-5,0) -- cycle;
		\draw[->] (0,0) -- (10,0) node [below] {$x=(x_1,x_2) $};
		\draw[->] (5,0) -- (5,5)  node [right] {$y$}; 
		\draw [red, xshift=5cm] plot [smooth, tension=0.7] coordinates { (-5,4) (-3,3) (0,0.5) (3,3) (5,4)};
		\draw [blue, xshift=5cm] plot [smooth, tension=0] coordinates {(-5,0) (-5,4)};
		\draw [blue, xshift=5cm] plot [smooth, tension=0] coordinates {(5,4) (5,0)};
		\node[xshift=5cm] at (4.5,0.5) {$\Omega$};
		\node[red, xshift=5cm] at (3,4.5) {Deformed electrode $A$, $y=w(x)$};
		\node[blue, xshift=5cm, align=left] at (-4.5, 3) {bounding\\ box $C$};
		\node[blue, xshift=5cm, align=left] at (5.5, 3) { $C$};
		\node[black, xshift=5cm] at (0, -0.5) {fixed electrode $B$, $y=0$};
		\draw[-{Latex[length=2mm]}, thick, xshift=5cm] (-4,1) arc[start angle=-120, end angle=-90, radius=5cm] node [at end, right]{};
		\draw[-{Latex[length=2mm]}, thick, xshift=5cm] (-4,2.2) arc[start angle=90, end angle=30, radius=3cm] node [at end, right]{};
		\node[black, xshift=5cm] at (-3, 1.2) {fluid velocity $\mathbf{v}$};
		\draw[-{Latex[length=2mm]}, thick, xshift=5cm] (4,2) arc[start angle=100, end angle=130, radius=5cm] node [at end, left]{};
		\draw[-{Latex[length=2mm]}, thick, xshift=5cm] (4,1) arc[start angle=-60, end angle=-90, radius=5cm] node [at end, left]{};
\end{tikzpicture}\caption{\small{Sketch of a MEMS capacitor undergoing sparking. The domain $\Omega$ is depicted by the shaded region. Its boundary decomposes into a Dirichlet portion, formed by electrodes $A$ and $B$, where the electric potential $\phi$ together with the concentrations $p$ and $n$ are prescribed, and a Neumann portion $C =\partial\Omega_N $, where the artificial boundary is placed far enough away that homogeneous Neumann conditions can reasonably be imposed.}}\normalsize
  \label{MEMSsketch}	\end{figure}

\subsection{Physical Motivation and Related Work}
Micro-electro-mechanical systems (MEMS) capacitive switches and varactors are critical components in modern radio-frequency (RF) communication systems, offering low power consumption, high isolation, and excellent linearity \cite{rebeiz2004rf, van2012capacitive}. However, their long-term reliability is severely compromised by dielectric charging and discharging phenomena \cite{papaioannou2005temperature, yuan2005modeling}. During device operation, charge carriers are injected from the actuation electrode into the dielectric layer under high electric fields, where they become trapped at defect sites and interfaces \cite{herfst2008kelvin}. The subsequent release and transport of these trapped charges alter the electrostatic force balance, causing drift in the pull-in voltage and eventual device failure \cite{molinero2006dielectric}.

A predictive understanding of these degradation mechanisms requires a self-consistent continuum model that couples the electric field distribution with the transport, generation, and recombination of charge carriers. When the mean free path of carriers is much smaller than the characteristic device dimensions—as is typical for solid-state dielectrics under moderate electric fields—the drift-diffusion approximation provides an appropriate level of physical description \cite{selberherr1984analysis, van1950theory}.

The classical drift-diffusion model, originally formulated by Van Roosbroeck \cite{van1950theory} for semiconductor devices, has been extensively studied both analytically and numerically over the past seven decades \cite{markowich1985stationary,markowich2012semiconductor,jungel2001quasi}. Under isothermal conditions, it consists of Poisson's equation for the electrostatic potential, coupled with continuity equations for electron and positive ion densities. The model has been successfully applied to the simulation of p-n junctions, metal-oxide-semiconductor field-effect transistors (MOSFETs), solar cells,  a variety of other semiconductor devices \cite{brezzi2005discretization,miller1999application,vasileska2017computational}, and further gas discharge \cite{gimperlein2026analysis, islamov2003global, islamov2006regularity}.

The mathematical analysis and numerical approximation of stationary drift-diffusion systems of \eqref{S_EPSys}-type pose several interconnected challenges that have motivated extensive research at the interface of nonlinear partial differential equations, semiconductor physics, and computational science. A central difficulty is the strong nonlinear coupling: the electric potential $\phi$ is determined by the space charge distribution through Poisson's equation, while the carrier transport equations contain drift terms proportional to $\nabla \phi$, and this bidirectional coupling is further intensified by the exponential sensitivity of the ionisation source to the local field magnitude \cite{markowich1985stationary, van1950theory}. Convection dominance and boundary layers constitute another major challenge. When the modulus of the gas flow velocity $v$ or the electric field $\nabla \phi$ is large compared to the diffusion coefficients (which are not equal to the mobility coefficients), the transport equations become convection-dominated, a regime characterized by sharp boundary layers near downstream boundaries where the solution varies on a length scale proportional to the diffusion coefficient divided by the drift velocity \cite{markowich1985stationary,roos2008robust}. In such regimes, standard centred finite difference or Galerkin finite element discretizations produce spurious oscillations unless appropriately stabilized \cite{brooks1982streamline,hughes2018multiscale}. A further essential requirement is the positivity of carrier densities: the physical densities $p$ and $n$ must remain non-negative, yet classical discretization methods do not inherently guarantee this property, particularly in convection-dominated regimes or when the ionization source term becomes negative (attachment regime). Violations of positivity lead to physically meaningless solutions and can cause divergence of nonlinear solvers \cite{miller1999application,patankar2018numerical}. Finally, a non-Lipschitz nonlinearity arises from the exponential gate function 
$e^{-\alpha_2/|\nabla\phi|}$, which is not
Lipschitz continuous at $|\nabla\phi|=0$. This singularity requires regularization for both the existence analysis and the numerical implementation, typically by introducing a small parameter $\delta > 0$ and replacing $|\nabla\phi|$
 with $E_{\delta}(|\nabla\phi|)=\sqrt{|\nabla\phi|^2 +\delta^2}$ \cite{abdel2025existence,chainais2003finite}.

The drift-diffusion model has been the subject of intense analytical investigation since the seminal work of Van Roosbroeck \cite{van1950theory}. Existence and uniqueness of weak solutions for the basic semiconductor equations were established by Gajewski and Gröger \cite{gajewski1986basic}, Markowich \cite{markowich1985stationary}, and Frehse and Naumann \cite{frehse1993existence}, among others. These results rely on Schauder fixed-point arguments, monotone operator theory, and entropy-type estimates to obtain uniform bounds on the carrier densities. Comprehensive treatments can be found in the monographs by Markowich \cite{markowich1985stationary}, Markowich, Ringhofer, and Schmeiser \cite{markowich2012semiconductor}, and Jüngel \cite{jungel2001quasi}.

The inclusion of impact ionisation terms brings additional analytical difficulties. In the context of semiconductor devices, ionization rates were analysed by Chen et al. \cite{degond2004note} and Degond et al. \cite{degond2000numerical} for hydrodynamic and energy-transport models. For corona discharge problems, Budd \cite{budd1991coronas} provided a detailed analysis of the space charge problem in a concentric-cylinder geometry, employing a hard cutoff ionisation gate and a photoionisation boundary condition at the ionisation boundary. His study revealed the existence of free boundaries separating ionised and non-ionised regions, as well as oscillatory instabilities in the transient regime. More recently, Abdel et al. \cite{abdel2025existence} proved existence and uniform bounds for a stationary drift-diffusion system with generation and both ionic and electronic carriers, employing a truncation–regularisation approach combined with Stampacchia-type estimates.

On the numerical side, the Scharfetter–Gummel exponentially fitted finite-volume scheme \cite{scharfetter2005large} and stabilised finite element formulations—SUPG \cite{brooks1982streamline}, Galerkin least-squares and variational multiscale methods \cite{hughes2018multiscale}, and discontinuous Galerkin schemes \cite{chen2020steady, liu2016analysis} remain standard tools for capturing boundary layers in the drift-diffusion system \cite{brezzi2005discretization,kumar2017fem}. Positivity is commonly enforced through Slotboom \cite{miller1999application}, quasi-Fermi \cite{jungel2001quasi,jungel2001positivity}, or logarithmic \cite{brezzi1989two,polak1987semiconductor} transformations of the dependent variables; Pérez-Escudero et al. \cite{perez2025comparison} recently compared the primitive and logarithmic formulations together with Gummel \cite{gummel2005self} and Newton–Raphson solvers, finding that the logarithmic SUPG formulation yields oscillation-free solutions on coarse meshes. We adopt a shifted-logarithmic SUPG formulation, detailed in Section \ref{app2}.

\subsection{Overview and Structure of the Paper}
This work provides the first rigorous existence proof for the stationary Townsend-discharge system \eqref{S_EPSys}--\eqref{bdyval}, together with a tailored finite element solver. Our contributions are twofold.
\ 

While existence theory for standard drift-diffusion systems is classical \cite{frehse1993existence, frehse1996stationary, gajewski1986basic, markowich1985stationary}, the Townsend-type impact ionization source introduces an exponential activation gate $e^{-\alpha_2/|\nabla\phi|}$ that is non-Lipschitz and singular at $|\nabla\phi| = 0$. Prior analyses of corona discharge rely on hard cutoffs or simplified geometries \cite{budd1991coronas}. We prove, for the first time, the existence of weak solutions to the full stationary system with uniform upper and lower bounds on the carrier densities. The proof combines a double regularization (truncation of the source and mollification of the densities), monotone operator theory (Browder–Minty), and Stampacchia-type truncation arguments to control the exponential field dependence—extending the classical Gajewski–Gröger–Markowich framework \cite{gajewski1986basic, markowich1985stationary} and drawing on the recent estimates of Abdel et al. \cite{abdel2025existence}.

\ 
We develop and validate a shifted-logarithmic finite element solver integrating three features: 
 an algebraic positivity-preserving transformation $p = p_{\min} + e^{P}$, $n = n_{\min} + e^{N}$, reflecting the a priori $L^{\infty}$ bounds of Theorem \ref{SS_Existence} and Theorem \ref{apriori_est_thm}; 
  consistent SUPG stabilization \cite{brooks1982streamline, hughes2018multiscale} suppressing spurious oscillations in convection-dominated regimes without violating positivity;  
  a robust Gummel–Newton outer–inner iteration stable across a wide range of Péclet numbers and geometric singularities. 
To our knowledge, this is the first fully stabilized, positivity-preserving scheme for the stationary Townsend model that is simultaneously backed by a rigorous existence proof. Validation includes manufactured-solution convergence (optimal $\mathrm{O}(h^2)$ in $L^2$, $\mathrm{O}(h)$ in $H^1$), mesh-independent iteration counts, and consistent behaviour across voltage sweeps and refinements, even at re-entrant corners and Dirichlet discontinuities where the field is singular.

For the proof of Theorem \ref{SS_Existence}, we shall make regularity assumptions on the domain, boundary data and function $\mathbf{v}$.
We define the space of test functions vanishing on the Dirichlet boundary as
\[H_D^1(\Omega)=\left\{u\in H^1(\Omega):\quad u=0\quad \text{a.e.~on}\quad\partial\Omega_D\right\}\] 
equipped with the standard $H^1(\Omega)$ norm. Owing to the Poincaré inequality for mixed boundary conditions \cite{gilbarg1998elliptic} the seminorm $||\nabla u ||_{L^2(\Omega)}$ is equivalent to the full $H^1(\Omega)$ norm on $H_D^1(\Omega)$. Throughout the analysis, $c_0$ denotes a generic positive constant whose value may change from line to line but depends only on the domain $\Omega$ and the dimension $d$.
\begin{assum}\label{assumption}
 Suppose that $d\in\{2,3\}$ and $\Omega$ is a bounded domain in $\mathbb{R}^d$  with boundary $\partial\Omega=\overline{\partial\Omega_D\cup\partial\Omega_N}$ and measure $\mathrm{meas}~\Omega$. The disjoint open subsets $\partial\Omega_D$, $\partial\Omega_N\subset\partial\Omega$ are of class $C^{1,1}$ and enclose an angle $\leq\pi/2$. Set $0<\epsilon_\phi<1$.

  Let  $\mathbf{v}$, $p_D$, $n_D$ and $\phi_D$ be given such that $\phi_D\in H^2(\Omega)$ and  $p_D$, $n_D \in W^{2,\sigma}(\Omega)$, $\sigma\geq3$, satisfying 
	\begin{enumerate}[label=(A\arabic*),ref=(A\arabic*)]
		\item\label{con1} $p_D\geq0$, $n_D\geq 0$ almost everywhere in ${\Omega}$ and 
		\[R_a\leq \inf_{\partial\Omega_D}p_D,~\inf_{\partial\Omega_D}n_D,\quad \sup_{\partial\Omega_D}p_D,~\sup_{\partial\Omega_D}n_D\leq R_b,\]
       \[K_1:=\frac{4c_0^3}{\epsilon_\phi}\left(\|\nabla p_D\|_{L^3(\Omega)}+\|\nabla n_D\|_{L^3(\Omega)}\right)\leq \frac{1}{32},\]
		\[{\nu\cdot\nabla p_D=\nu\cdot\nabla n_D=0\quad\text{on}\quad\partial\Omega_N}\]
		for some positive constants $R_a \leq R_b$; 
		\item\label{con2}  the velocity $\mathbf{v}\in  H^{2}(\Omega)$, $\nabla\cdot\mathbf{v}=0$ and 
        $\|\nabla|\mathbf{v}|\|_{L^2(\Omega)}\leq {1}/{\left(8c_0^2\left({1}/{\epsilon_p}+{1}/{\epsilon_n}\right)^2\right)}$.
        \item \label{con3}  there exists a $\tau_0>0$, such that if 
$\|\nabla \phi_D\|_{L^s(\Omega)}\leq \tau_0$,  $\|\mathbf{v}\|_{H^2(\Omega)}\leq\tau_0$,
then $$2\left(c_7+1\right)\left(2c_5^2+1\right)<\frac{1}{2},$$
where  
\[\begin{split}
c_7:=&\max\left\{\frac12,~ \frac{9}{50}+\frac{5\sqrt{30}}{108}\kappa_0^3+\frac{c_1}{3}\left(\frac{2c_1}{3}+\frac{1}{2}\right)\right\}\\
c_5:=&\left(\frac{c_1}{2}\left\||\nabla\phi_D|^2\right\|_{L^{\widetilde{\sigma}}(\Omega)}+\|\nabla\phi_D\|_{L^2(\Omega)}+c_4\right)\bigg/\min\left\{\frac{1}{4},\frac{1}{\epsilon_\phi}-\frac{1}{2}\right\},\qquad\qquad\quad\\
c_4:=&\widetilde{c}_4(\|\mathbf{v}\|_{H^2(\Omega)}+\|\nabla\phi_D\|_{L^2(\Omega)})\|\mathbf{v}\|_{H^2(\Omega)},\\
\widetilde{c}_4:=& \max\left\{{c_0^3}/{\epsilon_p^2},~ {c_0^3}/{\epsilon_n^2},~{c_0^2}/{\epsilon_p},~ {c_0^2}/{\epsilon_n}\right\}.\end{split}\] 
	\end{enumerate}
\end{assum}

\begin{defn}
   A \underline{weak solution} of the problem \eqref{S_EPSys}-\eqref{bdyval} is a triple $(\phi, p, n)$ of functions such that 
	\begin{subequations}\label{SS}
		\begin{equation}\label{SS_pos}
			0\leq p,~ n\leq K=\mathrm{const.}<\infty,\quad a.e.\quad\text{in}\quad\Omega,
		\end{equation}
		\begin{equation}\label{SS_wk_phi}
			\int_\Omega \epsilon_\phi\nabla\phi\cdot\nabla\varphi\mathrm{d}x=\int_\Omega \left(p-n\right)\varphi\mathrm{d}x,\quad\forall~\varphi\in V,
		\end{equation}
		\begin{equation}\label{SS_wk_p}
			\int_\Omega \left(\epsilon_p\nabla p+p\left(\epsilon_p\nabla \phi-\mathbf{v}\right)\right)\cdot\nabla\varphi\mathrm{d}x=\int_\Omega F\left(n,|\nabla\phi|\right)\varphi\mathrm{d}x,\quad\forall~\varphi\in V,
		\end{equation}
		\begin{equation}\label{SS_wk_n}
			\int_\Omega \left(\epsilon_n\nabla n-n\left(\epsilon_n\nabla \phi+\mathbf{v}\right)\right)\cdot\nabla\varphi\mathrm{d}x=\int_\Omega F\left(n,|\nabla\phi|\right)\varphi\mathrm{d}x,\quad\forall~\varphi\in V,
		\end{equation}
		\begin{equation}\label{SS_bdy}
			\phi(x)=\phi_D(x),\quad p(x)=p_D,\quad n(x)=n_D,\quad x\in\partial\Omega_D.
		\end{equation}
	\end{subequations}
\end{defn}
The paper is organized as follows.  In Section \ref{sec3}, we prove Theorem \ref{SS_Existence} via the regularization–approximation scheme, with the passage to the limit eliminating the parameters $\delta$ and $\epsilon$. Section \ref{numerical} presents the numerical validation on two-dimensional domains (unit square, L-shaped corner, annular corona) and a three-dimensional axisymmetric domed cylinder mimicking a MEMS capacitive switch.  The shifted-log discretization and the Gummel–Newton algorithm are detailed thereafter in Appendix \ref{app2}.



\section{Existence of Weak Solutions: Proof of Theorem \ref{SS_Existence}}\label{sec3}
The analysis of the system \eqref{S_EPSys} will rely on the abstract theory of elliptic differential equations (see, for example, \cite{fichera1965linear}, \cite{gilbarg1998elliptic} and \cite{ladyzhenskaya1968linear}), on functional analysis (see \cite{krasnoselskij1984geometrical} and \cite{rabinowitz1971some}) and on the basic stationary semiconductor device equations, i.e. Chapter 2 of \cite{markowich1985stationary}, Chapter 3 of \cite{markowich2012semiconductor}, references \cite{frehse1993existence}, \cite{frehse1994existence}, and  \cite{frehse1996stationary}.

This section is organized as follows. Subsection~\ref{Pre} collects the auxiliary results used throughout the proof: Sobolev embeddings, the H\"older inequality, and a standard elliptic estimate. The proof of Theorem~\ref{SS_Existence} then occupies Subsections~\ref{approx}--\ref{pass2limit}. In Subsection~\ref{approx} we construct the $\varepsilon$-approximate solutions $(\phi_\varepsilon, p_\varepsilon, n_\varepsilon)$ to \eqref{S_EPSys} by introducing an inner regularisation parameter $\delta$ and passing $\delta \to 0$; in Subsection~\ref{aprioriestimates} we derive the $\varepsilon$-uniform a priori estimates of Theorem~\ref{apriori_est_thm}; and in Subsection~\ref{pass2limit} we pass to the limit $\varepsilon \to 0$ to obtain a weak solution of \eqref{S_EPSys}.

\subsection{Preliminaries}\label{Pre}
For a real-valued measurable function $f$ on $\Omega$, we denote its positive and negative parts by
\[f^+=\max\{f,~ 0\},\quad f^-=\min\{f,~0\}.\] Clearly, $f = f^+ + f^-$ and $|f| = f^+ - f^-$. These truncation operations are essential in the construction of test functions for the entropy-type estimates developed in Section \ref{sec3}, particularly for establishing non-negativity of the carrier densities and for the Stampacchia truncation argument that yields the uniform upper bound. 

Note that  $H_D^1(\Omega)\times H_D^1(\Omega)$ is a Hilbert space with respect to the scalar product 
\[\langle(u,v), (\xi,\zeta)\rangle_{H^1\times H^1}=\int_{\Omega}\left(\nabla u\cdot\nabla\xi+\nabla v\cdot\nabla\zeta\right)\mathrm{d}x,\]
 $\|\cdot\|_{H^1\times H^1}$ is an associated norm, and 
 $\langle(u^*, v^*), (u,v)\rangle_{(H^1\times H^1)\to (H^1\times H^1)^{*}}$ represents the value of $(u^*, v^*)\in \left[H^1_D(\Omega)\times H^1_D(\Omega)\right]^{*}$ at $(u,v)$, where $\left[H_D^1(\Omega)\times H_D^1(\Omega)\right]^{*}$ is the dual space of $H_D^1(\Omega)\times H_D^1(\Omega)$.
 
From the Sobolev embedding theorem \cite{gilbarg1998elliptic, evans2022partial}, we recall the following fundamental estimate:
\begin{equation}\label{Sebd}
	\left(\int_\Omega \left|u\right|^q\mathrm{d}x\right)^{1/q}\leq c_0\left(\int_\Omega\left|\nabla u\right|^2\mathrm{d}x\right)^{1/2},\quad\forall~u\in H_D^1(\Omega),
\end{equation}
where $q\in[1, \infty)$ if $d=2$ and $q\in[1, 6]$ if $d=3$. The constant $c_0$ 
depends on $q$ and on the measure of $\Omega$ but is independent of the particular function $u$.

Considering $u$, $v\in H_D^1(\Omega)$ subject to the boundary value conditions  
\[u(x)=p_D(x)\quad\text{and}\quad v(x)=n_D(x)\quad \forall~x\in\partial\Omega_D,\]
where $p_D$ and $n_D$ are boundary values from  \eqref{SS_bdy} and satisfy Assumption \ref{assumption}, \eqref{Sebd} implies 
    \begin{equation}\label{Sobebd_pn}
        \begin{split}
            \|u+v\|_{L^\sigma(\Omega)}\leq c_1\left(1+\|\nabla(u+v)\|_{L^2(\Omega)}\right),
        \end{split}
    \end{equation}
    where $\sigma\in[1,\infty)$ if $\dim \Omega=2$, $\sigma\in[1,6]$ if $\dim\Omega=3$, and $c_1$ is a positive constant depending on $\sigma$, $\mathrm{meas}~\Omega$, $\|p_D\|_{H^1(\Omega)}$, and $\|n_D\|_{H^1(\Omega)}$. 

Moreover, by fixing $v \in L^6(\Omega)$, 	$w \in L^2(\Omega)$ and choosing $u \in L^3(\Omega)$, or $u \in L^{6/5}(\Omega)$,  for the $d=2$ and $d=3$ cases, respectively, it follows that
\begin{equation}\label{Holder-Cauchy}
			\begin{split}
				\int_\Omega uvw \, \mathrm{d}x\leq \|u\|_{L^3(\Omega)}\|v\|_{L^6(\Omega)}\|w\|_{L^2(\Omega)}&\leq \varepsilon\|v\|_{L^6(\Omega)}^2+\frac{1}{4\varepsilon}\|w\|_{L^2(\Omega)}^2\|u\|_{L^3(\Omega)}^2,\\
				\int_\Omega uv \, \mathrm{d}x\leq \|u\|_{L^{6/5}(\Omega)}\|v\|_{L^6(\Omega)}&\leq \varepsilon\|v\|_{L^6(\Omega)}^2+\frac{1}{4\varepsilon}\|u\|_{L^{6/5}(\Omega)}^2.
			\end{split}
		\end{equation}
		Further, we have the continuous embedding 
		\begin{equation}\label{Sob_emb_H2}
			H^2(\Omega)\hookrightarrow C^{0, \alpha}(\overline{\Omega})\hookrightarrow L^\infty(\Omega),\quad \|u\|_{L^\infty(\Omega)},~\|u\|_{C^{0, \alpha}(\Omega)}\leq c_0\|u\|_{H^2(\Omega)}
			\end{equation} with $\alpha>0$.
            
{Due to Assumption \ref{assumption} on the geometry of $\Omega$, 
we have the following standard result
from the theory of linear elliptic boundary value problems:} 
\begin{prop}\label{EllipticRegEst} 	 Let $f\in L^2(\Omega)$, $\phi_D\in H^{2}(\Omega)$. Then the Dirichlet problem $-\Delta u=f$ in $\Omega$  admits a unique solution $u\in H^{2}(\Omega)$ with $u-\phi_D \in H_D^{1}(\Omega)$, and the following estimate holds:
	\begin{equation}\label{elliptic_reg}
		\|u\|_{H^{2}(\Omega)}\leq c_0\left(\|f\|_{L^2(\Omega)}+\|\phi_D\|_{H^{2}(\Omega)}\right).
	\end{equation}
	
\end{prop}
\begin{rem}
    Note that, letting $(p,n)\in L^2(\Omega)\times L^2(\Omega)$ and $\delta>0$, there  exists a unique solution $\phi\in H^2(\Omega)$ of the equation
\begin{equation}\label{int_form_phi}\begin{split}				\int_\Omega\epsilon_\phi\nabla\phi\cdot\nabla\varphi\mathrm{d}x&=\int_\Omega\left( \frac{p^+}{1+\delta p^+}- \frac{n^+}{1+\delta}\right)\varphi\mathrm{d}x,\quad\forall~\varphi\in H_D^1(\Omega),\\ \phi&=\phi_D,~\text{a.e. on}~\partial\Omega_D,
		\end{split}	
\end{equation}
	and 
	\begin{equation}\label{elliptic_est}\|\nabla\phi\|_{L^2(\Omega)}\leq (c_0/\epsilon_\phi) \left[(2/\delta)+\|\nabla\phi_D\|_{L^2(\Omega)}\right].\end{equation}
\end{rem}

\subsection{Approximate Solutions}\label{approx}

We first study the existence of an approximate weak solution to the problem \eqref{S_EPSys}, Theorem \ref{E_ASS_Exist}, the proof replies on the theory of monotone operators in combination with \textit{a priori} estimates. 

    Let $\varepsilon>0$ and Define 
    \begin{equation}\label{nonlinearities}
            \begin{split}  
                F^\varepsilon_p\left(p^+, n^+,|\nabla\phi|\right)&=\frac{\epsilon_n|\nabla\phi|n^+\left[\alpha_1e^{-\alpha_2/|\nabla\phi|}
                -{\eta_0p^+}/{(\varepsilon+ p^+)}\right]}{1+\varepsilon\left|\epsilon_n|\nabla\phi|n^+\left[\alpha_1e^{-\alpha_2/|\nabla\phi|}
                -{\eta_0p^+}/{(\varepsilon+ p^+)}\right]\right|}\\
                F^\varepsilon_n\left( n^+,|\nabla\phi|\right)&=\frac{\epsilon_n|\nabla\phi|n^+\left(\alpha_1e^{-\alpha_2/|\nabla\phi|}
                -\eta_0\right)}{1+\varepsilon\left|\epsilon_n|\nabla\phi|n^+\left(\alpha_1e^{-\alpha_2/|\nabla\phi|}
                -\eta_0\right)\right|}.
            \end{split}
        \end{equation}
By using the theory of monotone operator we obtain Lemma \ref{mono_thy}. 
\begin{lem}\label{mono_thy}
There exists a pair $(\bar{p}, \bar{n})\in H^1(\Omega)\times H^1(\Omega)$ satisfying 
\begin{equation}\label{Ini_approx_SS_wk_p}
			\begin{split}\int_\Omega  \left(\epsilon_p\nabla \bar{p}+\frac{p^+}{1+\delta p^+}\left(\epsilon_p\nabla \phi-\mathbf{v}\right)\right)\cdot\nabla\varphi_p\mathrm{d}x
            =\int_\Omega F^\varepsilon_p\left(p^+, n^+,|\nabla\phi|\right)\varphi_p\mathrm{d}x,\end{split}
		\end{equation}
		\begin{equation}\label{Ini_approx_SS_wk_n}
			\begin{split}\int_\Omega  \left(\epsilon_n\nabla \bar{n}- \frac{n^+}{1+\delta n^+}\left(\epsilon_n\nabla\phi+\mathbf{v}\right)\right)\cdot\nabla\varphi_n\mathrm{d}x
            =\int_\Omega F^\varepsilon_n\left( n^+,|\nabla\phi|\right)\varphi_n\mathrm{d}x,\end{split}
		\end{equation}
		\begin{equation}\label{Ini_approx_SS_bdy}
			 \bar{p}(x)=p_D(x),\quad \bar{n}(x)=n_D(x),\quad x\in\partial\Omega_D,
		\end{equation} for all $\varphi_p,~\varphi_n\in H_D^1(\Omega)$, where $\phi$ is given by \eqref{int_form_phi} accordingly.
        
\end{lem}
    \begin{proof}
    We define a mapping $\mathcal{A}: \left[H_D^1(\Omega)\times H_D^1(\Omega)\right] \to   \left[H_D^1(\Omega)\times H_D^1(\Omega)\right]^*$ by
	\[\begin{split}
		\langle\mathcal{A}(p,n), (\xi, \zeta)\rangle_{(H^1\times H^1)\to (H^1\times H^1)^*}
        =\int_\Omega\left(\epsilon_p\nabla p\cdot\nabla\xi+\epsilon_n\nabla n\cdot\nabla\zeta\right)\mathrm{d}x,\end{split}\]
	where $\phi$ is given by \eqref{int_form_phi} accordingly, and a functional by 
    \[\begin{split}
		\langle\mathcal{F}, (\xi, \zeta)\rangle_{(H^1\times H^1)\to (H^1\times H^1)^*}
        =&\int_\Omega\left(\epsilon_n\frac{n^+}{1+\delta n^+}\nabla\phi\cdot\nabla\xi-\epsilon_p \frac{p^+}{1+\delta p^+}\nabla\phi\cdot\nabla\zeta\right)\mathrm{d}x\\
        +&\int_\Omega\left(\frac{p^+}{1+\delta p^+}\mathbf{v}\cdot\nabla\xi+\frac{n^+}{1+\delta n^+}\mathbf{v}\cdot\nabla\zeta\right)\mathrm{d}x\\
        +&\int_\Omega\left( F_p^\varepsilon\left(p^+, n^+,|\nabla\phi|\right)\xi+F_n^\varepsilon\left(n^+,|\nabla\phi|\right)\zeta\right)\mathrm{d}x,\end{split}\]
    
    We suppose that the sequence $(p_k, n_k)$ strongly converges to  $(p,n)$ in $H_D^1(\Omega)\times H_D^1(\Omega)$, i.e.
    \begin{equation}\label{cvg_uv}\lim_{k\to\infty}\|(p_k, n_k)-(p, n)\|_{H^1\times H^1}= 0.\end{equation}
    We shall prove 
    \begin{equation}\label{cvg_A}\limsup_{\substack{k \to \infty \\ \|(\xi,\zeta)\|=1}}\left|\langle\mathcal{A}(p_k,n_k)-\mathcal{A}(p,n), (\xi, \zeta)\rangle_{(H^1\times H^1)\to (H^1\times H^1)^*}\right|=0.\end{equation}
    Indeed, \eqref{cvg_uv} implies 
    \[\lim_{k\to\infty}\|p_k-p\|_{L^2(\Omega)}=0,\quad \lim_{k\to\infty}\|n_k-n\|_{L^2(\Omega)}=0,\]
    \[\lim_{k\to\infty}\|\nabla p_k- \nabla p\|_{L^2(\Omega)}=0,\quad \lim_{k\to\infty}\|\nabla n_k- \nabla n\|_{L^2(\Omega)}=0,\]
    \[\lim_{k\to\infty} p_k(x)=p(x),\quad \lim_{k\to\infty} n_k(x)=n(x),\quad\text{a.e. in}~\Omega.\]
    Then it follows that 
    \[\begin{split}
    &\left|\langle\mathcal{A}(p_k,n_k)-\mathcal{A}(p,n), (\xi, \zeta)\rangle_{(H^1\times H^1)\to (H^1\times H^1)^*}\right|\\
=&\left|\int_\Omega\left(\epsilon_p\nabla\left(p_k-p\right)\cdot\nabla\xi+\epsilon_n\nabla\left(n_k-n\right)\cdot\nabla\zeta\right)\mathrm{d}x\right|\\ 
\leq&\epsilon_p\|\nabla p_k-\nabla p\|_{L^2(\Omega)}\|\nabla\xi\|_{L^2(\Omega)}+\epsilon_n\|\nabla n_k-\nabla n\|_{L^2(\Omega)}\|\nabla\zeta\|_{L^2(\Omega)}\\
\to& 0\quad\text{as}\quad k\to\infty.
    \end{split}\]
    This concludes the proof of \eqref{cvg_A}.
    Moreover, \eqref{cvg_A} implies that $\mathcal{A}$ maps bounded sets into bounded set and $\mathcal{A}$ is hemicontinuous.

    Since 
    \[\begin{split}&\langle\mathcal{A}(p_1,n_1)-\mathcal{A}(p_2,n_2), (p_1-p_2, n_1-n_2)\rangle\\
    \geq &\min\{\epsilon_p,\epsilon_n\}\left(\|\nabla(p_1-p_2)\|_{L^2(\Omega)}^2+\|\nabla(n_1-n_2)\|_{L^2(\Omega)}^2\right)\\
    \geq&  0\end{split}\]
     $\mathcal{A}$ is a monotone operator.

     Obviously,
     \[\lim_{\|(p,n)\|\to\infty}\frac{\langle\mathcal{A}(p,n), (p, n)\rangle}{\|(p,n)\|}=\infty,\]
     i.e. $\mathcal{A}$ is coercive. 

 We next show that 
     \begin{equation}\label{cvg_F}
         \langle\mathcal{F}, (\xi,\zeta)\rangle<\infty\quad\forall~ \xi,~\zeta\in H_D^1(\Omega).
     \end{equation}
     We note that 
     \[0\leq \frac{p^+}{1+\delta p^+}\leq \frac{1}{\delta},\quad 0\leq \frac{n^+}{1+\delta n^+}\leq \frac{1}{\delta}\]
     use \eqref{elliptic_est} and obtain 
     \[\frac{p^+}{1+\delta p^+}\nabla\phi,\quad \frac{n^+}{1+\delta n^+}\nabla\phi,\quad
     F_p^{\varepsilon}\left(p^+, n^+,|\nabla\phi|\right)\in L^2(\Omega).\]
     So using H\"older inequality and \ref{con2} of Assumption \ref{assumption}, it follows that 
    \[
    \begin{split}
        \left|\int_\Omega F_p^\varepsilon\left(p^+, n^+,|\nabla\phi|\right)\xi\mathrm{d}x\right|
        \leq\frac{\epsilon_n(\alpha_1+\eta_0)}{\delta}\|\nabla\phi\|_{L^2(\Omega)}\|\xi\|_{L^2(\Omega)}<\infty;
        \end{split}\] 
        \[\begin{split}		&\left|\int_\Omega\epsilon_p\left(\frac{p^+}{1+\delta p^+}\nabla\phi\right)\cdot\nabla\xi\mathrm{d}x\right|\leq \frac{\epsilon_p}{\delta}\|\nabla\phi\|_{L^2(\Omega)}\|\nabla\xi\|_{L^2(\Omega)}<\infty\\
        \end{split}\]
        and 
        \[\begin{split}
            &\left|\int_\Omega\left(\frac{p^+}{1+\delta p^+}\right)\mathbf{v}\cdot\nabla\xi\
            \mathrm{d}x\right|\leq \frac{1}{\delta}\|\mathbf{v}\|_{L^2(\Omega)}\|\nabla\xi\|_{L^2(\Omega)}<\infty
        \end{split}\]
        Similarly,
        \[\left|\int_\Omega F_n^\varepsilon\left(n^+,|\nabla\phi|\right)\xi\mathrm{d}x\right|<\infty,\quad 
        \left|\int_\Omega\epsilon_n\frac{n^+}{1+\delta n^+}\nabla\phi\cdot\nabla\zeta\mathrm{d}x\right|<\infty,\quad 
		\left|\int_\Omega\frac{n^+}{1+\delta n^+}\mathbf{v}\cdot\nabla\zeta\mathrm{d}x\right|<\infty.
    \]
    Hence we obtain \eqref{cvg_F}.

By using Browder-Minty Theorem (see, for example, Theorem 10.49 from \cite{renardy2004introduction} (P.364) and Theorem 2.1 from \cite{lions1969quelques} (P.171)), we obtain that there exists a $(\bar{u}, \bar{v})$ such that 
\[\mathcal{A}(\bar{u},\bar{v})=\mathcal{F},\]
equivalently,
\[\langle\mathcal{A}(\bar{u},\bar{v}), (\xi, 0)\rangle_{(H^1\times H^1)\to (H^1\times H^1)^*}=\langle\mathcal{F}, (\xi, 0)\rangle_{(H^1\times H^1)\to (H^1\times H^1)^*}\]
\[\langle\mathcal{A}(\bar{u},\bar{v}), (0, \zeta)\rangle_{(H^1\times H^1)\to (H^1\times H^1)^*}=\langle\mathcal{F}, (0, \zeta)\rangle_{(H^1\times H^1)\to (H^1\times H^1)^*}\]
    This concludes the proof of Lemma \ref{mono_thy}.
    \end{proof}
 Using Lemma \ref{mono_thy} in combination with Leray-Schauder fixed point theorem, we obtain
    \begin{lem}\label{ASS_Existence}For each $\delta>0$, there exist functions $\phi_\delta$, $p_\delta$, $n_\delta\in H^1(\Omega)$ such that
	\begin{subequations}\label{D_approx_SS}
		\begin{equation}\label{D_approx_SS_pos}
			0\leq p_\delta,~ n_\delta,\quad a.e.\quad\text{in}\quad\Omega,
		\end{equation}
		\begin{equation}\label{D_approx_SS_wk_phi}
			\int_\Omega \epsilon_\phi\nabla\phi_\delta\cdot\nabla\varphi\mathrm{d}x=\int_\Omega \left( \frac{p_\delta}{1+\delta p_\delta}- \frac{n_\delta}{1+\delta n_\delta}\right)\varphi\mathrm{d}x,\quad\forall~\varphi\in H_D^1(\Omega),
		\end{equation}
		\begin{equation}\label{D_approx_SS_wk_p}\begin{split}
			\int_\Omega \left(\epsilon_p\nabla p_\delta+ \frac{p_\delta}{1+\delta p_\delta}\left(\epsilon_p\nabla\phi_\delta-\mathbf{v}\right)\right)\cdot\nabla\varphi_p\mathrm{d}x
            =\int_\Omega F_p^\varepsilon\left(p_\delta, n_\delta,|\nabla\phi_\delta|\right)\varphi_p\mathrm{d}x,\end{split}
		\end{equation}
		\begin{equation}\label{D_approx_SS_wk_n}
			\begin{split}\int_\Omega  \left(\epsilon_n\nabla n_\delta- \frac{n_\delta}{1+\delta n_\delta}\left(\epsilon_n\nabla\phi_\delta+\mathbf{v}\right)\right)\cdot\nabla\varphi_n\mathrm{d}x
            =\int_\Omega F_n^\varepsilon\left( n_\delta, |\nabla\phi_\delta|\right)\varphi_n\mathrm{d}x,
		\end{split}\end{equation}
		\begin{equation}\label{D_approx_SS_bdy}
			\phi_\delta(x)=\phi_D(x),\quad p_\delta(x)=p_D,\quad n_\delta(x)=n_D,\quad x\in\partial\Omega_D,
		\end{equation}
	\end{subequations}
    for all $\varphi_p,~\varphi_n\in H_D^1(\Omega)$.
    \end{lem}
    \begin{proof}
    According to the equations \eqref{Ini_approx_SS_wk_p}-\eqref{Ini_approx_SS_wk_n} from Lemma \ref{ASS_Existence},
    we define a maps $\mathcal{T}$ by
    \begin{equation}\label{opt_T}\begin{split}
        \mathcal{T}:~\{p,n\}&\longmapsto\mathcal{T}(\{p,n\})=\{\bar{p},\bar{n}\},\\ L^2(\Omega)\times L^2(\Omega)&\longrightarrow L^2(\Omega)\times L^2(\Omega).\end{split}
    \end{equation}
         We shall show the existence of a fixed point of the map $\mathcal{T}$. To begin with, we take $\varphi_p=(\bar{p}-p_D)/\epsilon_p$ in equation \eqref{Ini_approx_SS_wk_p} and $\varphi_n=(\bar{n}-n_D)/\epsilon_n$ in equation \eqref{Ini_approx_SS_wk_n}, add the two equations and  obtain
         \begin{equation}\label{eqns_sum}\begin{split}
             &\int_\Omega \left(\left|\nabla (\bar{p}-p_D)\right|^2+\left|\nabla (\bar{n}-n_D)\right|^2\right)\mathrm{d}x\\
             =&\int_\Omega \frac{p^+}{\epsilon_p(1+\delta p^+)}\left(\mathbf{v}-\epsilon_p\nabla\phi\right)\cdot\nabla\left(\bar{p}-p_D\right)\mathrm{d}x\\
			+&\int_\Omega \frac{n^+}{\epsilon_n(1+\delta n^+)}\left(\mathbf{v}+\epsilon_n\nabla\phi\right)\cdot\nabla\left(\bar{n}-n_D\right)\mathrm{d}x\\
			+&\int_\Omega \left(F_p^\varepsilon\left(p^+, n^+, |\nabla\phi|\right)\frac{\bar{p}-p_D}{\epsilon_p}+F_n^\varepsilon\left(n^+, |\nabla\phi|\right)\frac{\bar{n}-n_D}{\epsilon_n}\right)\mathrm{d}x.\\
			-&\int_\Omega\left(\frac{\nabla p_D\cdot\nabla\left(\bar{p}-p_D\right)}{\epsilon_p}+\frac{\nabla n_D\cdot\nabla\left(\bar{n}-n_D\right)}{\epsilon_n}\right)\mathrm{d}x\\
            =:&I_1+I_2+I_3+I_4.
         \end{split}\end{equation}
         Using H\"older inequality, the elliptic equation \eqref{int_form_phi} and estimate \eqref{elliptic_est}, we obtain
         \begin{equation}\label{I_est}\begin{split}
             |I_1+I_2|&\leq\frac{1}{\delta}\left[\left(\frac{1}{\epsilon_p}+\frac{1}{\epsilon_n}\right)\|\mathbf{v}\|_{L^2(\Omega)}+\frac{c_0}{\epsilon_\phi}\left(\frac{2}{\delta}+\|\nabla\phi_D\|_{L^2(\Omega)}\right)\right]\\
             &\times\left(\|\nabla(\bar{p}-p_D)\|_{L^2(\Omega)}+\|\nabla(\bar{n}-n_D)\|_{L^2(\Omega)}\right),\\
             |I_3|&\leq \frac{1}{\varepsilon}\left(\|\nabla(\bar{p}-p_D)\|_{L^2(\Omega)}+\|\nabla(\bar{n}-n_D)\|_{L^2(\Omega)}\right),\\
             |I_4|&\leq\left(\|\nabla p_D\|_{L^2(\Omega)}/\epsilon_p+\|\nabla n_D\|_{L^2(\Omega)}/\epsilon_n\right)
             \left(\|\nabla(\bar{p}-p_D)\|_{L^2(\Omega)}+\|\nabla(\bar{n}-n_D)\|_{L^2(\Omega)}\right).
         \end{split}\end{equation}
         \eqref{eqns_sum}, \eqref{I_est} and Cauchy inequality implies 
         \begin{equation}\label{barpn_apriori_est}
             \int_\Omega \left(\left|\bar{p}\right|^2+\left| \bar{n}\right|^2+\left|\nabla \bar{p}\right|^2+\left|\nabla \bar{n}\right|^2\right)\mathrm{d}x\leq C_\delta.
         \end{equation}
         Here $C_\delta$ is a positive constant depending on $\delta$ but independent of $\{p, n\}$. 

         We next prove that the map $\mathcal{T}$ is continuous from $L^2(\Omega)\times L^2(\Omega)$ into itself. In fact, we let $\{p_j, n_j\}\in L^2(\Omega)\times L^2(\Omega)$, $j=1,2,\dots,$ and set $\{p_j,n_j\}$ converges strongly to $\{p, n\}$ in $L^2(\Omega)\times L^2(\Omega)$ as $j\to\infty$.  We further let $\phi_j$, $\phi\in H^1(\Omega)$ be defined according to equation \eqref{int_form_phi} and suppose 
         \[\{\bar{p}_j, \bar{n}_j\}=\mathcal{T}\left(\{p_j, n_j\}\right),\quad \{\bar{p}, \bar{n}\}=\mathcal{T}\left(\{p, n\}\right),\quad j=1,2,\dots.\]
         Obviously, $\nabla \phi_j\to\nabla \phi$ strongly in $L^2(\Omega)$. Then there exists a subsequence $\{j_k\}$ such that 
         \[\nabla\phi_{j_k}\to\nabla\phi\quad\text{a.e. in}\quad \Omega,\]
         \[\bar{p}_{j_k}\to u,\quad \bar{n}_{j_k}\to v\quad\text{weakly in}\quad H^1(\Omega)\quad\text{and a.e. in}\quad\Omega\]
         as $k\to\infty$ (c.f. \eqref{barpn_apriori_est}). Set  $(\xi, \zeta)\in H^1(\Omega)\times H^1(\Omega)$ satisfying $\xi=p_D$ and $\zeta=n_D$ a.e. on $\partial\Omega_D$. For any $t>0$, by using \eqref{Ini_approx_SS_wk_p} and \eqref{Ini_approx_SS_wk_n}, it follows that 
         \[\begin{split}
             &\int_\Omega \frac{p^+}{1+\delta p^+}\left(\mathbf{v}-\epsilon_p\nabla\phi\right)\cdot\nabla(u-\xi)\mathrm{d}x+\int_\Omega F_p^\varepsilon\left(p^+, n^+,|\nabla\phi|\right)(u-\xi)\mathrm{d}x\\
             =&\epsilon_p\lim_{j\to\infty}\int_\Omega \nabla \bar{p}_{j_k}\cdot\nabla(\bar{p}_{j_k}-\xi)\mathrm{d}x\\
             \geq&\epsilon_p\liminf_{k\to\infty}\int_\Omega \nabla \bar{p}_{j_k}\cdot\nabla(u-\xi)\mathrm{d}x\\
             \geq&\frac{\epsilon_p}{t}\liminf_{k\to\infty}\int_\Omega\nabla(u+t(\xi-u))\cdot\nabla(\bar{p}_{j_k}-u-t(\xi-u))\mathrm{d}x\\
             =&\epsilon_p\int_\Omega\nabla(u+t(\xi-u))\cdot\nabla(u-\xi)\mathrm{d}x.
         \end{split}\]
         Letting $t\to 0$ ans setting $\xi=u- \varphi_p$ ($\varphi_p\in H_D^1(\Omega)$) imply 
         \[\begin{split}
         \epsilon_p\int_\Omega\nabla u\cdot\nabla\varphi_p\mathrm{d}x
         =
             \int_\Omega \frac{p^+}{1+\delta p^+}\left(\mathbf{v}-\epsilon_p\nabla\phi\right)\cdot\nabla\varphi_p\mathrm{d}x+\int_\Omega F_p^\delta\left(p^+, n^+,|\nabla\phi|\right)\varphi_p\mathrm{d}x.
         \end{split}\]
         Hence, $\nabla u=\nabla\bar{p}$ a.e. in $\Omega$, and therefore $u=\bar{p}$ a.e. in $\Omega$. By the same reasoning, $v=\bar{n}$ a.e. in $\Omega$. Thus, $\{\bar{p}_j, \bar{n}_j\}\to\{p,n\}$ strongly in $L^2(\Omega)\times L^2(\Omega)$.

         According to estimate \eqref{barpn_apriori_est}, it is readily seen that the image $\mathcal{T}(L^2(\Omega)\times L^2(\Omega))$ is precompact. So using Leray-Schauder fixed point theorem, it follows that  there exists a pair $\{p, n\}$ such that $\mathcal{T}(\{p, n\})=\{p, n\}$. Then one can determine $\phi\in H^1(\Omega)$ satisfying \eqref{int_form_phi} and $\phi=\phi_D$ a.e. on $\partial\Omega_D$.
         
         Inserting $\varphi_p=\min\{p, 0\}$ and $\varphi_n=\min\{n, 0\}$ into \eqref{Ini_approx_SS_wk_p} and \eqref{Ini_approx_SS_wk_n} with $\bar{p}$ replaced by $p$, $\bar{n}$ replaced by $n$ respectively, one gets $p\geq 0$ and $n\geq 0$ a.e. in $\Omega$.

         Consequently,  for each $\delta>0$, there exists a triple $\{\phi_\delta, p_\delta, n_\delta\}\in H^1(\Omega)\times H^1(\Omega)\times H^1(\Omega)$ which satisfies \eqref{D_approx_SS_pos}-\eqref{D_approx_SS_bdy}.

         This concludes the proof of Lemma \ref{ASS_Existence}. 
    \end{proof}
    \begin{thm}\label{E_ASS_Exist}
        For each $\varepsilon>0$, there exist functions $\phi_\varepsilon$, $p_\varepsilon$, $n_\varepsilon\in H^1(\Omega)$ such that
	\begin{subequations}\label{E_approx_SS}
		\begin{equation}\label{E_approx_SS_pos}
			0\leq p_\varepsilon,~ n_\varepsilon,\quad a.e.\quad\text{in}\quad\Omega,
		\end{equation}
		\begin{equation}\label{E_approx_SS_wk_phi}
			\int_\Omega \epsilon_\phi\nabla\phi_\varepsilon\cdot\nabla\varphi\mathrm{d}x=\int_\Omega \left( p_\varepsilon- n_\varepsilon\right)\varphi\mathrm{d}x,\quad\forall~\varphi\in H_D^1(\Omega),
		\end{equation}
		\begin{equation}\label{E_approx_SS_wk_p}\begin{split}
			\int_\Omega \left(\epsilon_p\nabla p_\varepsilon+ p_\varepsilon\left(\epsilon_p\nabla\phi_\varepsilon-\mathbf{v}\right)\right)\cdot\nabla\varphi_p\mathrm{d}x
            =\int_\Omega F^\varepsilon\left( n_\varepsilon,|\nabla\phi_\varepsilon|\right)\varphi_p\mathrm{d}x,\end{split}
		\end{equation}
		\begin{equation}\label{E_approx_SS_wk_n}
			\begin{split}\int_\Omega  \left(\epsilon_n\nabla n_\varepsilon- n_\varepsilon\left(\epsilon_n\nabla\phi_\varepsilon+\mathbf{v}\right)\right)\cdot\nabla\varphi_n\mathrm{d}x
            =\int_\Omega F^\varepsilon\left( n_\varepsilon, |\nabla\phi_\varepsilon|\right)\varphi_n\mathrm{d}x,
		\end{split}\end{equation}
		\begin{equation}\label{E_approx_SS_bdy}
			\phi_\varepsilon(x)=\phi_D(x),\quad p_\varepsilon(x)=p_D,\quad n_\varepsilon(x)=n_D,\quad x\in\partial\Omega_D.
		\end{equation}
	\end{subequations}
    for all $\varphi_p,~\varphi_n\in H_D^1(\Omega)$. Here
    \[F^\varepsilon\left( n_\varepsilon,|\nabla\phi|\right)=\frac{\epsilon_n|\nabla\phi|n_\varepsilon\left(\alpha_1e^{-\alpha_2/|\nabla\phi_\varepsilon|}
                -\eta_0\right)}{1+\varepsilon\left|\epsilon_n|\nabla\phi_\varepsilon|n_\varepsilon\left(\alpha_1e^{-\alpha_2/|\nabla\phi_\varepsilon|}
                -\eta_0\right)\right|}.\]
    \end{thm}
    \begin{proof} We divide two steps to prove the theorem.

    \
    
        \noindent\textbf{Step 1. A Prori Estimates for $\phi_\delta$, $p_\delta$ and $n_\delta$.}
To show a priori estimates for $\{p_\delta,~ n_\delta,~ \phi_\delta\}$.  Firstly,
	we obtain, from \eqref{D_approx_SS_wk_phi}, \eqref{elliptic_reg} from Proposition \ref{EllipticRegEst},  
	\begin{equation}\label{est_approx_SS_wk_phi}\|\nabla\phi_\delta\|_{L^2(\Omega)}\leq (c_0/\epsilon_\phi)\left( \left\|p_\delta\right\|_{L^2(\Omega)}+\left\|n_\delta\right\|_{L^2(\Omega)}\right)+c_0\left\|\nabla\phi_D\right\|_{L^2(\Omega)}.\end{equation} 
We insert $\varphi_p=\left(p_\delta-p_D\right)/\epsilon_p$ into \eqref{D_approx_SS_wk_p}, $\varphi_n=\left(n_\delta-n_D\right)/\epsilon_n$ into \eqref{D_approx_SS_wk_n}, add the two equations and obtain 
\begin{equation}\label{Int_est_whole}
    \begin{split}
        &\int_\Omega\left(|\nabla( p_\delta-p_D)|^2+|\nabla (n_\delta-n_D)|^2\right)\mathrm{d}x\\
        =&\int_\Omega\left[\nabla\phi_\delta\cdot\left(\frac{n_\delta\nabla(n_\delta-n_D)}{1+\delta n_\delta}-\frac{p_\delta\nabla(p_\delta-p_D)}{1+\delta p_\delta}\right)\right]\mathrm{d}x\\
        +&\int_\Omega\left[\left(\frac{1}{\epsilon_p}+\frac{1}{\epsilon_n}\right)\mathbf{v}\cdot\left(\frac{n_\delta\nabla(n_\delta-n_D)}{1+\delta n_\delta}+\frac{p_\delta\nabla(p_\delta-p_D)}{1+\delta p_\delta}\right)\right]\mathrm{d}x\\
        +&\int_\Omega \frac{F_p^\varepsilon\left(p_\delta, n_\delta,|\nabla\phi_\delta|\right)}{\epsilon_p}\left( p_\delta- p_D\right)\mathrm{d}x
        +\int_\Omega\frac{F_n^\varepsilon\left( n_\delta, |\nabla\phi_\delta|\right)}{\epsilon_n}\left( n_\delta- n_D\right)\mathrm{d}x\\
        -&\int_\Omega\left[\nabla p_D\cdot\nabla (p_\delta-p_D)+\nabla n_D\cdot\nabla( n_\delta- n_D)\right]\mathrm{d}x
    \end{split}
\end{equation}    
Thus, from Cauchy inequality,  we have the last three integrals on the right hand side of \eqref{Int_est_whole} can be estimated by
\begin{equation}\label{Int_est_last_three}
    \begin{split}
       &\left|\int_\Omega\left[\nabla p_D\cdot\nabla (p_\delta-p_D)+\nabla n_D\cdot\nabla( n_\delta- n_D)\right]\mathrm{d}x\right|\\
       \leq &\frac{1}{8}\int_\Omega\left(|\nabla( p_\delta-p_D)|^2+|\nabla (n_\delta-n_D)|^2\right)\mathrm{d}x+2\int_\Omega\left(|\nabla p_D|^2+|\nabla n_D|^2\right)\mathrm{d}x\\
       &\left|\int_\Omega \frac{F_p^\varepsilon\left(p_\delta, n_\delta,|\nabla\phi_\delta|\right)}{\epsilon_p}\left( p_\delta- p_D\right)\mathrm{d}x
        +\int_\Omega\frac{F_n^\varepsilon\left( n_\delta, |\nabla\phi_\delta|\right)}{\epsilon_n}\left( n_\delta- n_D\right)\mathrm{d}x\right|\\
        \leq&\frac{1}{8}\int_\Omega\left(|\nabla( p_\delta-p_D)|^2+|\nabla (n_\delta-n_D)|^2\right)\mathrm{d}x+2c_0^2\left(\frac{1}{\epsilon_p^2}+\frac{1}{\epsilon_n^2}\right)\frac{1}{\varepsilon^2}.
    \end{split}
\end{equation}  
    We define
    \[\Phi_\delta(p):=\int_0^p\frac{\tau}{1+\delta\tau}\mathrm{d}\tau,\quad p\in [0, \infty).\]
Clearly, 
\begin{equation}\label{Phi_del}
    \begin{split}0\leq\Phi_\delta(t)\leq \frac{p^2}{1+\delta p},\ \left(\Phi_\delta '(p)-\Phi_\delta '(\tilde{p})\right)\left(\Phi_\delta (p)-\Phi_\delta (\tilde{p})\right)\geq 0,\ \forall\ p,~ \tilde{p}\in [0,\infty),\\
\nabla\Phi_\delta(p)=\frac{p}{1+\delta p}\nabla p,\quad \forall~p\in H^1(\Omega),\quad p\geq 0~ \text{a.e. in}~ \Omega.\quad\qquad\qquad\end{split}\end{equation}
We write  
\begin{equation}\label{Int_est_1st}\begin{split}
    &\int_\Omega\left[\nabla\phi_\delta\cdot\left(\frac{n_\delta\nabla(n_\delta-n_D)}{1+\delta n_\delta}-\frac{p_\delta\nabla(p_\delta-p_D)}{1+\delta p_\delta}\right)\right]\mathrm{d}x\\
    =&\int_\Omega\left[\nabla\phi_\delta\cdot\left(\frac{n_\delta\nabla n_\delta}{1+\delta n_\delta}-\frac{n_D\nabla n_D}{1+\delta n_D}-\left(\frac{p_\delta\nabla p_\delta}{1+\delta p_\delta}-\frac{p_D\nabla p_D}{1+\delta p_D}\right)\right)\right]\mathrm{d}x\\
    +&\int_\Omega\left[\nabla\phi_\delta\cdot\left(\frac{n_D\nabla n_D}{1+\delta n_D}-\frac{p_D\nabla p_D}{1+\delta p_D}\right)\right]\mathrm{d}x
    -\int_\Omega\left[\nabla\phi_\delta\cdot\left(\frac{n_\delta\nabla n_D}{1+\delta n_\delta}-\frac{p_\delta\nabla p_D}{1+\delta p_\delta}\right)\right]\mathrm{d}x.
\end{split}\end{equation}
As the function $\varphi=\Phi_\delta(n)-\Phi_\delta(n_D)-\left(\Phi_\delta(p)-\Phi_\delta(p_D)\right)$ is admissible in \eqref{D_approx_SS_wk_phi}, we estimate the first integral on the right hand side of \eqref{Int_est_1st} as follows:
\begin{equation}\label{Int_est_1st-1}\begin{split}
    &\int_\Omega\left[\nabla\phi_\delta\cdot\left(\frac{n_\delta\nabla n_\delta}{1+\delta n_\delta}-\frac{n_D\nabla n_D}{1+\delta n_D}-\left(\frac{p_\delta\nabla p_\delta}{1+\delta p_\delta}-\frac{p_D\nabla p_D}{1+\delta p_D}\right)\right)\right]\mathrm{d}x\\
    =&\frac{1}{\epsilon_\phi}\int_\Omega\left\{\left(\Phi'_\delta(p_\delta)-\Phi'_\delta(n_\delta)\right)\left[\Phi_\delta(n_\delta)-\Phi_\delta(n_D)-(\Phi_\delta(p_\delta)-\Phi_\delta(p_D)\right]\right\}\mathrm{d}x\\
    \leq&\frac{1}{\epsilon_\phi}\int_\Omega\left\{\left[\Phi'_\delta(p_\delta)-\Phi'_\delta(n_\delta)\right]\left[\Phi_\delta(p_D)-\Phi_\delta(n_D)\right]\right\}\mathrm{d}x\\
    \leq&\frac{1}{\epsilon_\phi}\left\|\frac{p_\delta}{1+\delta p_\delta}-\frac{n_\delta}{1+\delta n_\delta}\right\|_{L^2(\Omega)}\left\|\Phi_\delta(p_D)-\Phi_\delta(n_D)\right\|_{L^2(\Omega)}\\
    \leq&\frac{1}{32}\int_\Omega\left(|\nabla p_\delta-p_D|^2+|\nabla n_\delta-n_D|^2\right)\mathrm{d}x+\frac{8c_0^2}{\epsilon_\phi^2}\left\|\Phi_\delta(p_D)-\Phi_\delta(n_D)\right\|_{L^2(\Omega)}^2\\
    +&\left(\|p_D\|_{L^2(\Omega)}+\|n_D\|_{L^2(\Omega)}\right)\left\|\Phi_\delta(p_D)-\Phi_\delta(n_D)\right\|_{L^2(\Omega)}.
\end{split}\end{equation}
To estimate the second integral on the right hand side of identity \eqref{Int_est_1st}, we
use \eqref{est_approx_SS_wk_phi}, H\"older inequality and Cauchy inequality, and obtain
\begin{equation}\label{Int_est_1st-2}\begin{split}
    &\int_\Omega\left[\nabla\phi_\delta\cdot\left(\frac{n_D\nabla n_D}{1+\delta n_D}-\frac{p_D\nabla p_D}{1+\delta p_D}\right)\right]\mathrm{d}x\\  \leq&\|\nabla\phi_\delta\|_{L^2(\Omega)}\left\|\frac{n_D\nabla n_D}{1+\delta n_D}-\frac{p_D\nabla p_D}{1+\delta p_D}\right\|_{L^2(\Omega)}\\
    \leq&\frac{c_0}{\epsilon_\phi}\left(\left\|p_\delta\right\|_{L^2(\Omega)}+\left\|n_\delta\right\|_{L^2(\Omega)}+\epsilon_\phi\|\phi_D\|_{L^2(\Omega)}\right)\left\|\frac{n_D\nabla n_D}{1+\delta n_D}-\frac{p_D\nabla p_D}{1+\delta p_D}\right\|_{L^2(\Omega)}\\
    \leq&\frac{1}{32}\int_\Omega\left(|\nabla p_\delta-p_D|^2+|\nabla n_\delta-n_D|^2\right)\mathrm{d}x+\frac{8c_0^4}{\epsilon_\phi^2}\left\|\frac{n_D\nabla n_D}{1+\delta n_D}-\frac{p_D\nabla p_D}{1+\delta p_D}\right\|_{L^2(\Omega)}^2\\
    +&c_0\left(\frac{\|p_D\|_{L^2(\Omega)}+\|n_D\|_{L^2(\Omega)}}{\epsilon_\phi}+\|\phi_D\|_{L^2(\Omega)}\right)\left\|\frac{n_D\nabla n_D}{1+\delta n_D}-\frac{p_D\nabla p_D}{1+\delta p_D}\right\|_{L^2(\Omega)}.
\end{split}\end{equation}
To estimate the third integral on on the right hand side of identity \eqref{Int_est_1st}, 
as $p_D$ and $n_D$ satisfy \ref{con1} of Assumption \ref{assumption},  
from \eqref{est_approx_SS_wk_phi}, H\"older inequality and Cauchy inequality, i.e. \eqref{Holder-Cauchy}, we obtain
\begin{equation}\label{Int_est_1st-3}\begin{split}
    &\int_\Omega\left[\nabla\phi_\delta\cdot\left(\frac{n_\delta\nabla n_D}{1+\delta n_\delta}-\frac{p_\delta\nabla p_D}{1+\delta p_\delta}\right)\right]\mathrm{d}x\\
    \leq&\|\nabla\phi_\delta\|_{L^2(\Omega)}\left(\|p_\delta\|_{L^6(\Omega)}+\| n_\delta\|_{L^6(\Omega)}\right)\left(\|\nabla p_D\|_{L^3(\Omega)}+\|\nabla n_D\|_{L^3(\Omega)}\right)\\
    \leq&\left(c_0/\epsilon_\phi\right)\left(\|p_\delta\|_{L^2(\Omega)}+\| n_\delta\|_{L^2(\Omega)}\right)\left(\|p_\delta\|_{L^6(\Omega)}+\| n_\delta\|_{L^6(\Omega)}\right)
    \left(\|\nabla p_D\|_{L^3(\Omega)}+\|\nabla n_D\|_{L^3(\Omega)}\right)\\
    +&c_0\|\phi_D\|_{L^2(\Omega)}\left(\|p_\delta\|_{L^6(\Omega)}+\| n_\delta\|_{L^6(\Omega)}\right)\left(\|\nabla p_D\|_{L^3(\Omega)}+\|\nabla n_D\|_{L^3(\Omega)}\right)\\
    \leq&\frac{c_0^3}{\epsilon_\phi}\left(\|\nabla p_D\|_{L^3(\Omega)}+\|\nabla n_D\|_{L^3(\Omega)}\right)\left(\|\nabla p_\delta\|_{L^2(\Omega)}+\| \nabla n_\delta\|_{L^2(\Omega)}\right)^2\\
    +&c_0^2\|\phi_D\|_{L^2(\Omega)}\left(\|\nabla p_D\|_{L^3(\Omega)}+\|\nabla n_D\|_{L^3(\Omega)}\right)\left(\|\nabla p_\delta\|_{L^2(\Omega)}+\|\nabla n_\delta\|_{L^2(\Omega)}\right)\\
    \leq&\left(K_1+\frac{1}{32}\right)
    \int_\Omega\left(|\nabla p_\delta-\nabla p_D|^2+|\nabla n_\delta-\nabla n_D|^2\right)\mathrm{d}x
    +\frac{K_1}{2}\left(\|\nabla p_D\|_{L^2(\Omega)}+\|\nabla n_D\|_{L^2(\Omega)}\right)^2\\
    +&8c_0^4\|\phi_D\|_{L^2(\Omega)}^2\left(\|\nabla p_D\|_{L^3(\Omega)}+\|\nabla n_D\|_{L^3(\Omega)}\right)^2\\
    \leq& \frac{1}{16}\int_\Omega\left(|\nabla p_\delta-\nabla p_D|^2+|\nabla n_\delta-\nabla n_D|^2\right)\mathrm{d}x+K_2.
\end{split}\end{equation}
Here 
\[\begin{split}
    K_2:=\frac{K_1}{2}\left(\|\nabla p_D\|_{L^2(\Omega)}+\|\nabla n_D\|_{L^2(\Omega)}\right)^2
    +8c_0^4\|\phi_D\|_{L^2(\Omega)}^2\left(\|\nabla p_D\|_{L^3(\Omega)}+\|\nabla n_D\|_{L^3(\Omega)}\right)^2.
\end{split}\]
Hence, combing \eqref{Int_est_1st}, \eqref{Int_est_1st-1}, \eqref{Int_est_1st-2} and \eqref{Int_est_1st-3}, we obtain 
\begin{equation}\label{Int_est_first-1}
    \begin{split}\int_\Omega\left[\nabla\phi_\delta\cdot\left(\frac{n_\delta\nabla(n_\delta-n_D)}{1+\delta n_\delta}-\frac{p_\delta\nabla(p_\delta-p_D)}{1+\delta p_\delta}\right)\right]\mathrm{d}x
    &\leq\frac{1}{8}\int_\Omega\left(|\nabla p_\delta-p_D|^2+|\nabla n_\delta-n_D|^2\right)\mathrm{d}x\\
    &+K_3,\end{split}
\end{equation}
where
\[\begin{split}
    K_3:=&\frac{8c_0^2}{\epsilon_\phi^2}\left\|\Phi_\delta(p_D)-\Phi_\delta(n_D)\right\|_{L^2(\Omega)}^2+\frac{8c_0^4}{\epsilon_\phi^2}\left\|\frac{n_D\nabla n_D}{1+\delta n_D}-\frac{p_D\nabla p_D}{1+\delta p_D}\right\|_{L^2(\Omega)}^2\\
    +&\left(\|p_D\|_{L^2(\Omega)}+\|n_D\|_{L^2(\Omega)}\right)\left\|\Phi_\delta(p_D)-\Phi_\delta(n_D)\right\|_{L^2(\Omega)}+K_2\\
    +&c_0\left(\frac{\|p_D\|_{L^2(\Omega)}+\|n_D\|_{L^2(\Omega)}}{\epsilon_\phi}+\|\phi_D\|_{L^2(\Omega)}\right)\left\|\frac{n_D\nabla n_D}{1+\delta n_D}-\frac{p_D\nabla p_D}{1+\delta p_D}\right\|_{L^2(\Omega)}.
\end{split}\]
Because $\mathbf{v}\in H^2(\Omega)$, $\nabla\cdot \mathbf{v}=0$ in $\Omega$ and $\mathbf{v}\cdot\nu=0$ on $\partial\Omega$, where $\nu$ is  unit outward normal vector on the boundary $\partial\Omega$, then integration by parts leads to 
\[\begin{split}
    \left(\frac{1}{\epsilon_p}+\frac{1}{\epsilon_n}\right)\int_\Omega\left[ \mathbf{v}\cdot\nabla\left(\Phi_\delta(n_\delta)+\Phi_\delta(p_\delta)\right)\right]\mathrm{d}x
    =&-\left(\frac{1}{\epsilon_p}+\frac{1}{\epsilon_n}\right)\int_\Omega\left[(\nabla\cdot \mathbf{v})\left(\Phi_\delta(n_\delta)+\Phi_\delta(p_\delta)\right)\right]\mathrm{d}x\\
    +&\left(\frac{1}{\epsilon_p}+\frac{1}{\epsilon_n}\right)\int_{\partial\Omega}\left[( \mathbf{v}\cdot\nu)\left(\Phi_\delta(n_\delta)+\Phi_\delta(p_\delta)\right)\right]\mathrm{d}S\\
    =&0,
\end{split}\]
so we have
\begin{equation}\label{Int_est_first-2}
    \begin{split}
        &\int_\Omega\left[\left(\frac{1}{\epsilon_p}+\frac{1}{\epsilon_n}\right)\mathbf{v}\cdot\left(\frac{n_\delta\nabla(n_\delta-n_D)}{1+\delta n_\delta}+\frac{p_\delta\nabla(p_\delta-p_D)}{1+\delta p_\delta}\right)\right]\mathrm{d}x\\
        =&\left(\frac{1}{\epsilon_p}+\frac{1}{\epsilon_n}\right)\int_\Omega\left[ \mathbf{v}\cdot\left(\frac{p_\delta\nabla p_D}{1+\delta p_\delta}+\frac{n_\delta\nabla n_D}{1+\delta n_\delta}\right)\right]\mathrm{d}x\\
        \leq&\frac{1}{8}\int_\Omega\left(|\nabla p_\delta-p_D|^2+|\nabla n_\delta-n_D|^2\right)\mathrm{d}x+K_4,
    \end{split}
\end{equation}
where 
\[\begin{split}K_4:=&2\left(\frac{c_0}{\epsilon_p}+\frac{c_0}{\epsilon_n}\right)^2\left(\|\mathbf{v}\cdot\nabla p_D\|_{L^2(\Omega)}+\|\mathbf{v}\cdot\nabla n_D\|_{L^2(\Omega)}\right)\\
+&\left(\frac{1}{\epsilon_p}+\frac{1}{\epsilon_n}\right)\int_\Omega\left(|\mathbf{v}|\left(|p_D\nabla p_D|+|n_D\nabla n_D|\right)\right)\mathrm{d}x.\end{split}\]
According to \eqref{Int_est_whole}, \eqref{Int_est_last_three}, \eqref{Int_est_first-1} and \eqref{Int_est_first-2}, we conclude 
\begin{equation}\label{H1_est}
    \begin{split}
        \int_\Omega\left(|\nabla( p_\delta-p_D)|^2+|\nabla (n_\delta-n_D)|^2\right)\mathrm{d}x
        \leq K_5(\varepsilon).
    \end{split}
\end{equation} 
Here 
\[K_5(\varepsilon):= 4\int_\Omega\left(|\nabla p_D|^2+|\nabla n_D|^2\right)\mathrm{d}x+4c_0^2\left(\frac{1}{\epsilon_p^2}+\frac{1}{\epsilon_n^2}\right)\frac{1}{\varepsilon^2}+2K_3+2K_4.\]
Next, we shall prove the boundedness of the densities $p_\delta$ and $n_\delta$ from above. We start the proof with the definitions of the following quantities:
\begin{equation}\label{k0}k_0:=\max\left\{\esssup_{\partial\Omega_D} p_D,\quad \esssup_{\partial\Omega_D} n_D\right\},\end{equation}
and, for any $k\geq k_0$, 
\[\begin{split}
    \Omega^p_{k,\delta}=\left\{x\in\Omega:~ p_\delta(x)>k\right\}=\left\{x\in\Omega,\ \Phi_\delta(p_\delta(x))>\Phi_\delta(k)\right\},\\
    \Omega^n_{k,\delta}=\left\{x\in\Omega:~ n_\delta(x)>k\right\}=\left\{x\in\Omega,\ \Phi_\delta(n_\delta(x))>\Phi_\delta(k)\right\},\\
    \omega_\delta(k)=\mathrm{meas}~\Omega^p_{k,\delta}+\mathrm{meas}~\Omega^n_{k,\delta}.\qquad\qquad\qquad
\end{split}\]
Clearly, the function $\omega$ is nonincreasing on $[k_0,\infty)$. By H\"older inequality and \eqref{Sebd} $(q=6)$, it follows that 
\begin{equation}\label{lower_bd}
    \begin{split}
        (h-k)\omega(h)&\leq\int_{\Omega^p_{k,\delta}}(p-k)\mathrm{d}x+\int_{\Omega^n_{k,\delta}}(n-k)\mathrm{d}x\\
        &\leq2c_0(\omega(k))^{5/6}\left(\int_{\Omega}\left(|\nabla (p-k)^+|^2+|\nabla (n-k)^+|^2\right)\mathrm{d}x\right)^{1/2}
    \end{split}
\end{equation}
for all $h> k\geq k_0$.

Inserting admissible test functions $\varphi_p=(p_\delta-k)^+/\epsilon_p$ and $\varphi_n=(n_\delta-k)^+/\epsilon_n$ with $k\geq k_0$ in \eqref{D_approx_SS_wk_p} and \eqref{D_approx_SS_wk_n} respectively, one obtains 
\begin{equation}\label{lowerbd_Int_est_whole}
    \begin{split}
        &\int_\Omega\left(|\nabla(p_\delta-k)^+|^2+|\nabla(n_\delta-k)^+|^2\right)\mathrm{d}x\\
        =&\int_\Omega\left[\nabla\phi_\delta\cdot\left(\frac{n_\delta\nabla(n_\delta-k)^+}{1+\delta n_\delta}-\frac{p_\delta\nabla(p_\delta-k)^+}{1+\delta p_\delta}\right)\right]\mathrm{d}x\\
        +&\left(\frac{1}{\epsilon_p}+\frac{1}{\epsilon_n}\right)\int_\Omega\left[\mathbf{v}\cdot\left(\frac{n_\delta\nabla(n_\delta-k)^+}{1+\delta n_\delta}+\frac{p_\delta\nabla(p_\delta-k)^+}{1+\delta p_\delta}\right)\right]\mathrm{d}x\\
        +&\int_\Omega\left( \frac{F_p^\varepsilon\left(p_\delta, n_\delta,|\nabla\phi_\delta|\right)}{\epsilon_p}\left( p_\delta- k\right)^+
        +\frac{F_n^\varepsilon\left( n_\delta, |\nabla\phi_\delta|\right)}{\epsilon_n}\left( n_\delta- k\right)^+\right)\mathrm{d}x. \\
    \end{split}
\end{equation}
According to \eqref{Phi_del}, we obtain the estimate for the first integral on the right hand side of \eqref{lowerbd_Int_est_whole}:
\begin{equation}\label{lowerbd_Int_est-1}
    \begin{split}
&\int_\Omega\left[\nabla\phi_\delta\cdot\left(\frac{n_\delta\nabla(n_\delta-k)^+}{1+\delta n_\delta}-\frac{p_\delta\nabla(p_\delta-k)^+}{1+\delta p_\delta}\right)\right]\mathrm{d}x\\
    =&\int_{\Omega_{k,\delta}^n}\left[\nabla\phi_\delta\cdot\nabla\Phi_\delta(n)\right]\mathrm{d}x-\int_{\Omega_{k,\delta}^p}\left[\nabla\phi_\delta\cdot\nabla\Phi_\delta(p)\right]\mathrm{d}x\\
    =&\frac{1}{\epsilon_\phi}\int_\Omega\left[\left(\frac{p_\delta}{1+\delta p_\delta}-\frac{n_\delta}{1+\delta n_\delta}\right)\left((\Phi_\delta(n_\delta)-\Phi_\delta(k))^+-(\Phi_\delta(p_\delta)-\Phi_\delta(k))^+\right)\right]\mathrm{d}x\\
    \leq&0.
\end{split}\end{equation}
According to \ref{con2} of Assumption \ref{assumption}, $\mathbf{v}\in H^2(\Omega)$, $\nabla\cdot \mathbf{v}=0$ in $\Omega$ and $\mathbf{v}\cdot\nu=0$ on $\partial\Omega$, with the unit outward normal vector $\nu$ on the boundary $\partial\Omega$, integration by parts leads to
\[\begin{split}
    &\left(\frac{1}{\epsilon_p}+\frac{1}{\epsilon_n}\right)\left[\int_{\Omega}\mathbf{v}\cdot\left(\nabla(\Phi_\delta(n)-\Phi_\delta(k))^++\nabla(\Phi_\delta(p)-\Phi_\delta(k))^+\right)\mathrm{d}x\right]\\
    =&-\left(\frac{1}{\epsilon_p}+\frac{1}{\epsilon_n}\right)\int_\Omega\left[(\nabla\cdot \mathbf{v})\left((\Phi_\delta(n_\delta)-\Phi_\delta(k))^++(\Phi_\delta(p_\delta)-\Phi_\delta(k))^+\right)\right]\mathrm{d}x\\
    +&\left(\frac{1}{\epsilon_p}+\frac{1}{\epsilon_n}\right)\int_{\partial\Omega}\left[( \mathbf{v}\cdot\nu)\left((\Phi_\delta(n_\delta)-\Phi_\delta(k))^++(\Phi_\delta(p_\delta)-\Phi_\delta(k))^+\right)\right]\mathrm{d}S\\
    =&0.
\end{split}\]
Hence the  second integral on the right hand side of \eqref{lowerbd_Int_est_whole} is zero due to 
\begin{equation}\label{lowerbd_Int_est-v}
    \begin{split}
        &\left(\frac{1}{\epsilon_p}+\frac{1}{\epsilon_n}\right)\int_\Omega\left[\mathbf{v}\cdot\left(\frac{n_\delta\nabla(n_\delta-k)^+}{1+\delta n_\delta}+\frac{p_\delta\nabla(p_\delta-k)^+}{1+\delta p_\delta}\right)\right]\mathrm{d}x\\
        = &\left(\frac{1}{\epsilon_p}+\frac{1}{\epsilon_n}\right)\left[\int_{\Omega_{k,\delta}^n}\mathbf{v}\cdot\nabla\Phi_\delta(n)\mathrm{d}x+\int_{\Omega_{k,\delta}^p}\mathbf{v}\cdot\nabla\Phi_\delta(p)\mathrm{d}x\right]\\
        =&\left(\frac{1}{\epsilon_p}+\frac{1}{\epsilon_n}\right)\left[\int_{\Omega}\mathbf{v}\cdot\left(\nabla(\Phi_\delta(n)-\Phi_\delta(k))^++\nabla(\Phi_\delta(p)-\Phi_\delta(k))^+\right)\mathrm{d}x\right]=0
    \end{split}
\end{equation}
Using \eqref{H1_est}, we obtain  the estimate for the third integral on the right hand side of \eqref{lowerbd_Int_est_whole}:
\begin{equation}\label{lowerbd_Int_est-2}
    \begin{split}
        &\int_\Omega\left( \frac{F_p^\varepsilon\left(p_\delta, n_\delta,|\nabla\phi_\delta|\right)}{\epsilon_p}\left( p_\delta- k\right)^+
        +\frac{F_n^\varepsilon\left( n_\delta, |\nabla\phi_\delta|\right)}{\epsilon_n}\left( n_\delta- k\right)^+\right)\mathrm{d}x\\
        \leq&\frac1\varepsilon\left(\frac{1}{\epsilon_p}+\frac{1}{\epsilon_n}\right)c_0K_5(\varepsilon)(\omega(k))^{1/2}.
    \end{split}
\end{equation}
Hence, \eqref{lowerbd_Int_est_whole}, \eqref{lowerbd_Int_est-1}, \eqref{lowerbd_Int_est-v} and \eqref{lowerbd_Int_est-2} imply 
\[\int_\Omega\left(|\nabla(p_\delta-k)^+|^2+|\nabla(n_\delta-k)^+|^2\right)\mathrm{d}x\leq \frac1\varepsilon\left(\frac{1}{\epsilon_p}+\frac{1}{\epsilon_n}\right)c_0K_5(\varepsilon)(\omega(k))^{1/2},\]
this, together with \eqref{lower_bd}, imply 
\[\omega(h)\leq \frac{K_6(\varepsilon)}{h-k}(\omega(k))^{13/12},\quad\forall\ h>k>k_0,\]
thus, \[\omega(k_0+K_7(\varepsilon))=0,\]
i.e.
\begin{equation}\label{uperbound_del}p_\delta\leq k_0+K_8(\varepsilon),\quad n_\delta\leq k_0+K_8(\varepsilon),\quad\text{a.e. in}\quad\Omega,\end{equation}
where $K_j(\varepsilon)$, $j=6,7,8$, are the constants depending on $\varepsilon$.
\color{black}

\

    \noindent\textbf{Step 2. Passage to limit $\delta\to 0$ in \eqref{D_approx_SS}.}
    By passing the subsequence, from \eqref{H1_est}, one concludes 
    \[\begin{split}
         p_\delta\to p_\varepsilon,~ n_\delta\to n_\varepsilon,~\phi_\delta\to\phi_\varepsilon~\text{weakly in}~H^1(\Omega),~\text{strongly in}~L^2(\Omega),\\
        p_\delta(x)\to p_\varepsilon(x),~ n_\delta(x)\to n_\varepsilon(x)~\text{for a. e. }~x\in\Omega,\qquad\qquad
    \end{split}\]
    as $\delta\to 0$. Obviously, from \eqref{D_approx_SS_pos}, \eqref{D_approx_SS_bdy} and \eqref{uperbound_del}, we have 
    \[\begin{split}
        0\leq p_\varepsilon,~ n_\varepsilon\leq k_0+K_8(\varepsilon),~\text{a.e. in }~\Omega,\\
        p_\varepsilon=p_D,~ n_\varepsilon=n_D,~\phi_\varepsilon=\phi_D,~\text{a.e. in}~\Omega,
    \end{split}\]
    i.e. \eqref{E_approx_SS_pos} and \eqref{E_approx_SS_bdy}.
    Furthermore, \eqref{D_approx_SS_wk_phi} implies 
    \[\begin{split}&\int_\Omega \epsilon_\phi|\nabla(\phi_\delta-\phi_\varepsilon)|^2\mathrm{d}x\\
    =&\int_\Omega \left( \frac{p_\delta}{1+\delta p_\delta}- \frac{n_\delta}{1+\delta n_\delta}\right)(\phi_\delta-\phi_\varepsilon)\mathrm{d}x-\int_\Omega(\nabla\phi_\varepsilon\cdot\nabla(\phi_\delta-\phi_\varepsilon))\mathrm{d}x\\
    \to& 0,\quad\text{as $\delta\to 0$.} \end{split}\]
     Using this and estimate \eqref{H1_est}, from \eqref{D_approx_SS_wk_p} and \eqref{D_approx_SS_wk_n}, we obtain by an analogous reasoning
    \[\nabla p_\delta\to\nabla p_\varepsilon,~\nabla n_\delta\to\nabla n_\varepsilon~\text{strongly in}~ L^2(\Omega)~\text{as}~\delta\to 0.\]
    Therefore, without loss of generality, we assume that $\nabla p_\delta\to\nabla p_\varepsilon$, $\nabla n_\delta\to\nabla n_\varepsilon$, and $\nabla \phi_\delta\to\nabla \phi_\varepsilon$ almost everywhere in $\Omega$ as $\delta\to0$. Consequently, the passage to limit $\delta\to 0$ in \eqref{D_approx_SS_wk_phi}-\eqref{D_approx_SS_wk_n} leads to  \eqref{E_approx_SS_wk_phi}-\eqref{E_approx_SS_wk_n}.

    This concludes the proof of Theorem \ref{E_ASS_Exist}.
    \end{proof}
\subsection{A Priori Estimates}\label{aprioriestimates}
\begin{thm} \label{apriori_est_thm}
    There exists a constant $C$,  depending on coefficients $\epsilon_{\phi}$, $\epsilon_p$, $\epsilon_n$, $\mu_{p}$, and boundary values $\phi_D$, $p_D$, $n_D$ such that the solution $(\phi_\varepsilon, p_\varepsilon, n_\varepsilon)$ from Theorem \ref{E_ASS_Exist} satisfies 
    \begin{subequations}\label{aprio_est}
        \begin{equation}
            \int_\Omega\left(|\nabla p_\varepsilon|^2+|\nabla n_\varepsilon|^2\right)\mathrm{d}x\leq C,\quad
            \int_\Omega\left(1+ p_\varepsilon+ n_\varepsilon\right)|\nabla\phi_\varepsilon|^2\mathrm{d}x\leq C,\end{equation}\begin{equation}\label{aprio_est2}
                p_\varepsilon\leq k_o+C,\quad n_\varepsilon\leq k_o+C,
            \end{equation} 
    \end{subequations}
    where $k_0$ is given by \eqref{k0}. 
\end{thm}
\begin{proof}
    According to Proposition \ref{EllipticRegEst}, we insert $\varphi=\phi_\varepsilon-\phi_D$ into \eqref{E_approx_SS_wk_phi}  and obtain 
    \begin{equation}\label{elliptic_est_eps}
        \begin{split}
            \|\nabla\phi_\varepsilon\|_{L^2(\Omega)}\leq c_2\left(\|p_\varepsilon-n_\varepsilon\|_{L^2(\Omega)}+\|\nabla\phi_D\|_{L^2(\Omega)}\right),
        \end{split}
    \end{equation} where the constant $c_2=c_0\max\left\{1/\epsilon_\phi, ~1\right\}$.

    We insert $\varphi_p=(\phi_\varepsilon-\phi_D)/\epsilon_p$ into \eqref{E_approx_SS_wk_p}, $\varphi_n=-(\phi_\varepsilon-\phi_D)/\epsilon_n$ into \eqref{E_approx_SS_wk_n}, add two resulting equations and obtain
    \begin{equation}\label{Int_est_eps_1}
        \begin{split}
            \int_\Omega\left(\left(p_\varepsilon+n_\varepsilon\right)\left|\nabla\phi_\varepsilon\right|^2\right)\mathrm{d}x
            =&\int_\Omega\left((p_\varepsilon+n_\varepsilon)\nabla\phi_\varepsilon\cdot\nabla\phi_D\right)\mathrm{d}x+\int_\Omega\left(\nabla(n_\varepsilon-p_\varepsilon)\cdot\nabla(\phi_\varepsilon-\phi_D)\right)\mathrm{d}x\\
            +&\int_\Omega\left(\left(\frac{p_\varepsilon}{\epsilon_p}-\frac{n_\varepsilon}{\epsilon_n}\right)\mathbf{v}\cdot\nabla(\phi_\varepsilon-\phi_D)\right)\mathrm{d}x
            =:I^\varepsilon_1+I^\varepsilon_2+I_3^\varepsilon.
        \end{split}
    \end{equation}
    Using \eqref{Sobebd_pn}, we infer that 
    \begin{equation}\label{I_eps_1}\begin{split}
        I_1^\varepsilon\leq& \frac{1}{2}\int_\Omega\left(p_\varepsilon+n_\varepsilon\right)|\nabla\phi_\varepsilon|^2\mathrm{d}x+\frac{1}{2}\int_\Omega\left(p_\varepsilon+n_\varepsilon\right)|\nabla\phi_D|^2\mathrm{d}x\\
         \leq &\frac{1}{2}\int_\Omega\left(p_\varepsilon+n_\varepsilon\right)|\nabla\phi_\varepsilon|^2\mathrm{d}x+\frac{1}{2}\|p_\varepsilon+n_\varepsilon\|_{L^\sigma(\Omega)}\left\||\nabla\phi_D|^2\right\|_{L^{\widetilde{\sigma}}(\Omega)}\\
        \leq&\frac{1}{2}\int_\Omega\left(p_\varepsilon+n_\varepsilon\right)|\nabla\phi_\varepsilon|^2\mathrm{d}x
        +\frac{c_1}{2}\left(1+\|\nabla(p_\varepsilon+n_\varepsilon)\|_{L^2(\Omega)}\right)\left\||\nabla\phi_D|^2\right\|_{L^{\widetilde{\sigma}}(\Omega)},
    \end{split}\end{equation}
    where $\left({1}/{\sigma}\right)+\left({1}/{\widetilde{\sigma}}\right)=1$ and $\sigma,~\widetilde{\sigma}\in (1, \infty)$.
    
To estimate $I_2^\varepsilon$, we use \eqref{E_approx_SS_wk_phi} with test function $\varphi=n_\varepsilon-p_\varepsilon-(n_D-p_D)$ and \eqref{elliptic_est_eps} and get 
\begin{equation}\label{I_eps_2}\begin{split}
    I_2^\varepsilon=&\frac{1}{\epsilon_\phi}\int_\Omega\left[\left(p_\varepsilon-n_\varepsilon\right)(n_\varepsilon-p_\varepsilon-(n_D-p_D))\right]\mathrm{d}x\\
    +&\int_\Omega\left[\nabla\phi_\varepsilon\cdot\nabla(n_D-p_D)\right]\mathrm{d}x-\int_\Omega\left(\nabla\phi_D\cdot\nabla(n_\varepsilon-p_\varepsilon)\right)\mathrm{d}x\\
    \leq&\left(1-\frac{1}{\epsilon_\phi}\right)\|n_\varepsilon-p_\varepsilon\|_{L^2(\Omega)}^2+\|\nabla(n_\varepsilon-p_\varepsilon)\|_{L^2(\Omega)}\|\nabla\phi_D\|_{L^2(\Omega)}+c_3,
\end{split}\end{equation}
    where $$c_3=\frac{1}{2}\int_\Omega|n_D-p_D|^2\mathrm{d}x+\frac{c_2^2}{2}\|\nabla(n_D-p_D)\|_{L^2(\Omega)}+c_2\|\nabla\phi_D\|_{L^2(\Omega)}\|\nabla(n_D-p_D)\|_{L^2(\Omega)}.$$ 
    Now we estimate $I_3^\varepsilon$ as follows: 
    \begin{equation}\label{I_eps_3}
        \begin{split}
            I_3^\varepsilon&\leq\max\left\{\frac{1}{\epsilon_p},\frac{1}{\epsilon_n}\right\}\int_\Omega\left(\left(p_\varepsilon+n_\varepsilon\right)|\mathbf{v}||\nabla\phi_\varepsilon|\right)\mathrm{d}x +\max\left\{\frac{1}{\epsilon_p},\frac{1}{\epsilon_n}\right\}\int_\Omega\left(\left(p_\varepsilon+n_\varepsilon\right)|\mathbf{v}|\nabla\phi_D\right)\mathrm{d}x\\
            &\leq \frac{1}{4}\int_\Omega\left((p_\varepsilon+n_\varepsilon)|\nabla\phi_\varepsilon|^2\right)\mathrm{d}x+\max\left\{\frac{1}{\epsilon_p^2},\frac{1}{\epsilon_n^2}\right\}\int_\Omega\left((p_\varepsilon+n_\varepsilon)|\mathbf{v}|^2\right)\mathrm{d}x\\
            &+\max\left\{\frac{1}{\epsilon_p},\frac{1}{\epsilon_n}\right\}\|p_\varepsilon+n_\varepsilon\|_{L^2(\Omega)}\|\mathbf{v}\|_{L^\infty(\Omega)}\|\nabla\phi_D\|_{L^2(\Omega)}\\
            &\leq \frac{1}{4}\int_\Omega\left((p_\varepsilon+n_\varepsilon)|\nabla\phi_\varepsilon|^2\right)\mathrm{d}x
            +c_4\left(\|\nabla p_\varepsilon\|_{L^2(\Omega)}+\|\nabla n_\varepsilon\|_{L^2(\Omega)}\right),
        \end{split}
    \end{equation}
 where \[\begin{split}c_4=\widetilde{c}_4(\|\mathbf{v}\|_{H^2(\Omega)}+\|\nabla\phi_D\|_{L^2(\Omega)})\|\mathbf{v}\|_{H^2(\Omega)},\quad 
\widetilde{c}_4= \max\left\{{c_0^3}/{\epsilon_p^2},~ {c_0^3}/{\epsilon_n^2},~{c_0^2}/{\epsilon_p},~ {c_0^2}/{\epsilon_n}\right\},\end{split}\] and we use continuous embedding \eqref{Sob_emb_H2}. 
Hence, with setting $\epsilon_\phi\in(0,1)$ from Assumption \ref{assumption}, \eqref{Int_est_eps_1}--\eqref{I_eps_3} imply
    \begin{equation}\label{aprio_est_eps}
        \begin{split}
            \int_\Omega\left((p_\varepsilon-n_\varepsilon)^2+(p_\varepsilon+n_\varepsilon)|\nabla\phi_\varepsilon|^2\right)\mathrm{d}x
            \leq c_5\left(\|\nabla p_\varepsilon\|_{L^2(\Omega)}+\|\nabla n_\varepsilon\|_{L^2(\Omega)}\right)+c_6.
        \end{split}
    \end{equation}
     Here 
     \[c_5=\left(\frac{c_1}{2}\left\||\nabla\phi_D|^2\right\|_{L^{\widetilde{\sigma}}(\Omega)}+\|\nabla\phi_D\|_{L^2(\Omega)}+c_4\right)\bigg/\min\left\{\frac{1}{4},\ \frac{1}{\epsilon_\phi}-1\right\},\]
     \[c_6=\frac{(c_1/2)\left\||\nabla\phi_D|^2\right\|_{L^{\widetilde{\sigma}}(\Omega)}+c_3}{\min\left\{({1}/{4}),\ ({1}/{\epsilon_\phi})-1\right\}}.\]

We insert $\varphi_p=(p_\varepsilon-p_D)/\epsilon_p$ into \eqref{E_approx_SS_wk_p} and $\varphi_n=(n_\varepsilon-n_D)/\epsilon_n$ into \eqref{E_approx_SS_wk_n}, and obtain
\begin{equation}\label{sum_of_J_k}\begin{split}
    &\int_\Omega\left(|\nabla p_\varepsilon|^2+|\nabla n_\varepsilon|^2\right)\mathrm{d}x
    \leq \int_\Omega\left(\nabla p_\varepsilon\cdot\nabla p_D+\nabla n_\varepsilon\cdot\nabla n_D\right)\mathrm{d}x\\
    &+\frac{1}{2}\int_\Omega\left(\nabla\phi_\varepsilon\cdot\nabla(n_\varepsilon^2-p_\varepsilon^2)\right|\mathrm{d}x
    +\int_\Omega|\nabla\phi_\varepsilon\cdot(p_\varepsilon\nabla p_D-n_\varepsilon\nabla n_D)|\mathrm{d}x\\
    &+\left(\frac{1}{2\epsilon_p}+\frac{1}{2\epsilon_n}\right)\int_\Omega\left|\mathbf{v}\cdot\nabla(p_\varepsilon^2+n_\varepsilon^2)\right|\mathrm{d}x
    +\int_\Omega\left|\mathbf{v}\cdot\left(\frac{p_\varepsilon\nabla p_D}{\epsilon_p}+\frac{n_\varepsilon\nabla n_D}{\epsilon_n}\right)\right|\mathrm{d}x\\
    &+\kappa_0\int_\Omega\left((p_\varepsilon+n_\varepsilon)|\nabla\phi_\varepsilon|(p_\varepsilon+n_\varepsilon+p_D+n_D)\right)\mathrm{d}x=:\sum_{k=1}^6 J_k, \end{split}
\end{equation}
where $\kappa_0=(\alpha_1+\eta_0)\epsilon_n\max\{1/\epsilon_p,~1/\epsilon_n\}$.

We estimate $J_k$, $k=1,\dots, 6$ as follows.  Clearly, 
\begin{equation}\label{J_1}
    \begin{split}
        J_1\leq \frac{1}{16}\int_\Omega\left(|\nabla p_\varepsilon|^2+|\nabla n_\varepsilon|^2\right)\mathrm{d}x+ 4\int_\Omega\left(|\nabla p_D|^2+|\nabla n_D|^2\right)\mathrm{d}x.
    \end{split}
\end{equation}
To estimate $J_2$, we use \eqref{E_approx_SS_wk_phi}, \eqref{elliptic_est_eps}, and obtain 
\begin{equation}\label{J_2}\begin{split}
    J_2&=\frac{1}{2}\left(\int_\Omega\frac{p_\varepsilon-n_\varepsilon}{\epsilon_\phi}\left(n_\varepsilon^2-p_\varepsilon^2-\left(n_D^2-p_D^2\right)\right)\mathrm{d}x+\int_\Omega\nabla\phi_\varepsilon\cdot\nabla(n_D^2-p_D^2)\mathrm{d}x\right)\\
    &\leq \frac{1}{2\epsilon_\phi}\int_\Omega(n_\varepsilon-p_\varepsilon)\left(n_D^2-p_D^2\right)\mathrm{d}x+\frac{1}{2}\int_\Omega\nabla\phi_\varepsilon\cdot\nabla\left(n_D^2-p_D^2\right)\mathrm{d}x\\
    &\leq \frac{1}{16}\int_\Omega\left(\left|\nabla p_\varepsilon\right|^2+\left|\nabla n_\varepsilon\right|^2\right)\mathrm{d}x+\frac{4c_0^2}{\epsilon_\phi^2}\left\|n_D^2-p_D^2\right\|_{L^2(\Omega)}^2 \\
    &+\left(4c_0^2c_2^2+\frac{c_2}{2}\|\nabla\phi_D\|_{L^2(\Omega)}\right)\left\|\nabla\left(n_D^2-p_D^2\right)\right\|_{L^2(\Omega)}.
\end{split}\end{equation}
Using \eqref{Sobebd_pn} gives
\begin{equation}\label{J_3}
    \begin{split}
        J_3\leq& \int_\Omega(p_\varepsilon+n_\varepsilon)|\nabla\phi_\varepsilon|(|\nabla p_D|+|\nabla n_D|)\mathrm{d}x\\
        \leq&\frac{1}{4}\left(\int_\Omega\left(p_\varepsilon+n_\varepsilon\right)^{s}\mathrm{d}x\right)^{1/s}\left(\int_\Omega(|\nabla p_D|+|\nabla n_D|)^{r}\mathrm{d}x\right)^{2/r}
        +\int_\Omega(p_\varepsilon+n_\varepsilon)\left|\nabla\phi_\varepsilon\right|^2\mathrm{d}x\\
        \leq&\frac{1}{16}\int_\Omega\left(|\nabla p_\varepsilon|^2+\left|\nabla n_\varepsilon\right|^2\right)\mathrm{d}x+\int_\Omega(p_\varepsilon+n_\varepsilon)\left|\nabla\phi_\varepsilon\right|^2\mathrm{d}x\\
        +&\frac{c_0^2}{2}\left(\int_\Omega(|\nabla p_D|+|\nabla n_D|)^{r}\mathrm{d}x\right)^{4/r}.
    \end{split}
\end{equation}
Here $s\in(0,1)$ and $r\in(0,2)$.

Then, by using \ref{con2} from Assumption \ref{assumption}, estimates \eqref{Sebd} and \eqref{Holder-Cauchy}, we obtain the estimations of $J_4$ and $J_5$:
\begin{equation}\label{J4J5}
    \begin{split}
        J_4\leq& \left(\frac{1}{\epsilon_p}+\frac{1}{\epsilon_n}\right)\||\mathbf{v}|\|_{L^6(\Omega)}\left(\|p_\varepsilon\|_{L^3(\Omega)}\|\nabla p_\varepsilon\|_{L^2(\Omega)}+\|n_\varepsilon\|_{L^3(\Omega)}\|\nabla n_\varepsilon\|_{L^2(\Omega)}\right)\\
        \leq&\left(\frac{c_0}{\epsilon_p}+\frac{c_0}{\epsilon_n}\right)\|\nabla|\mathbf{v}|\|_{L^2(\Omega)}\int_\Omega\left(\left|\nabla p_\varepsilon\right|^2+\left|\nabla n_\varepsilon\right|^2\right)\mathrm{d}x\\
        \leq&\frac{1}{8}\int_\Omega\left(\left|\nabla p_\varepsilon\right|^2+\left|\nabla n_\varepsilon\right|^2\right)\mathrm{d}x,\\
        J_5\leq&\left(\frac{1}{\epsilon_p}+\frac{1}{\epsilon_n}\right)\||\mathbf{v}|\|_{L^6(\Omega)}\left(\|p_\varepsilon\|_{L^2(\Omega)}\|\nabla p_D\|_{L^3(\Omega)}+\|n_\varepsilon\|_{L^2(\Omega)}\|\nabla n_D\|_{L^3(\Omega)}\right)\\
        \leq&\frac{1}{16}\int_\Omega\left(\left|\nabla p_\varepsilon\right|^2+\left|\nabla n_\varepsilon\right|^2\right)\mathrm{d}x\\
        +&\left(\frac{2c_0^2}{\epsilon_p}+\frac{2c_0^2}{\epsilon_n}\right)^2\|\nabla|\mathbf{v}|\|^2_{L^2(\Omega)}\left(\|\nabla p_D\|_{L^3(\Omega)}^2+\|\nabla n_D\|_{L^3(\Omega)}^2\right).
    \end{split}
\end{equation}
We remain to show the estimation of $J_6$. First, by using \eqref{Sobebd_pn} and Young's inequality, 
\begin{equation*}
    \begin{split}
        &\kappa_0\int_\Omega(p_\varepsilon+n_\varepsilon)^2\left|\nabla\phi_\varepsilon\right|\mathrm{d}x\\
        \leq&\kappa_0\left(\int_\Omega(p_\varepsilon+n_\varepsilon)\left|\nabla\phi_\varepsilon\right|^2\mathrm{d}x\right)^{1/2}\left(\int_\Omega(p_\varepsilon+n_\varepsilon)^3\mathrm{d}x\right)^{1/2}\\
        \leq&\frac{1}{16}\int_\Omega\left(\left|\nabla p_\varepsilon\right|^2+\left|\nabla n_\varepsilon\right|^2\right)\mathrm{d}x+c_7\left(1+\left(\int_\Omega(p_\varepsilon+n_\varepsilon)\left|\nabla\phi_\varepsilon\right|^2\mathrm{d}x\right)^2\right),
    \end{split}
\end{equation*}
where $$c_7=\max\left\{\frac12,~ \frac{9}{50}+\frac{5\sqrt{30}}{108}\kappa_0^3+\frac{c_1}{3}\left(\frac{2c_1}{3}+\frac{1}{2}\right)\right\},$$ 
and
\[\begin{split}
    &\kappa_0\int_\Omega(p_\varepsilon+n_\varepsilon)\left|\nabla\phi_\varepsilon\right|(p_D+n_D)\mathrm{d}x\\
    \leq&\frac{1}{16}\int_\Omega\left(\left|\nabla p_\varepsilon\right|^2+\left|\nabla n_\varepsilon\right|^2\right)\mathrm{d}x+\frac{1}{2}\left(\int_\Omega(p_\varepsilon+n_\varepsilon)\left|\nabla\phi_\varepsilon\right|^2\mathrm{d}x\right)^2+\frac{\kappa_0^4}{8}+8c_0^2\|p_D+n_D\|^4_{L^4(\Omega)},
\end{split}\]
thus we obtain 
\begin{equation}\label{J6}
    \begin{split}
        J_6\leq& \frac{1}{8}\int_\Omega\left(\left|\nabla p_\varepsilon\right|^2+\left|\nabla n_\varepsilon\right|^2\right)\mathrm{d}x+c_7\left(1+\left(\int_\Omega(p_\varepsilon+n_\varepsilon)\left|\nabla\phi_\varepsilon\right|^2\mathrm{d}x\right)^2\right)\\
        +&\frac{1}{2}\left(\int_\Omega(p_\varepsilon+n_\varepsilon)\left|\nabla\phi_\varepsilon\right|^2\mathrm{d}x\right)^2+\frac{\kappa_0^4}{8}+8c_0^2\|p_D+n_D\|^4_{L^4(\Omega)}.
    \end{split}
\end{equation}
Inserting \eqref{J_1}-\eqref{J6} into \eqref{sum_of_J_k} infers 
\begin{equation}\label{Int_est_gradpn}
    \begin{split}
        \int_\Omega\left(\left|\nabla p_\varepsilon\right|^2+\left|\nabla n_\varepsilon\right|^2\right)\mathrm{d}x&\leq 2M_1+2\left(c_7+1\right)\left(\int_\Omega(p_\varepsilon+n_\varepsilon)\left|\nabla\phi_\varepsilon\right|^2\mathrm{d}x\right)^2,
    \end{split}
\end{equation}
where
\[\begin{split}
    M_1&:=4\int_\Omega\left(|\nabla p_D|^2+|\nabla n_D|^2\right)\mathrm{d}x
        +\frac{4c_0^2}{\epsilon_\phi^2}\|n_D^2-p_D^2\|_{L^2(\Omega)}^2\\ 
    &+\left(4c_0^2c_2^2+\frac{c_2}{2}\|\nabla\phi_D\|_{L^2(\Omega)}\right)\left\|\nabla\left(n_D^2-p_D^2\right)\right\|_{L^2(\Omega)}\\
        &+\frac{c_0^2}{2}\left(\int_\Omega(|\nabla p_D|+|\nabla n_D|)^{r}\mathrm{d}x\right)^{4/r}+\frac{\kappa_0^4}{8}+8c_0^2\|p_D+n_D\|^4_{L^4(\Omega)}+c_7+\frac{1}{2}\\
        &+\left(\frac{2c_0^2}{\epsilon_p}+\frac{2c_0^2}{\epsilon_n}\right)^2\|\nabla|\mathbf{v}|\|_{L^2(\Omega)}^2\left(\|\nabla p_D\|^2_{L^3(\Omega)}+\|\nabla n_D\|^2_{L^3(\Omega)}\right).
\end{split}\]
We estimate the integral on the right hand side of \eqref{Int_est_gradpn} with the help of \eqref{aprio_est_eps} to obtain 
\[\begin{split}
    \int_\Omega\left(\left|\nabla p_\varepsilon\right|^2+\left|\nabla n_\varepsilon\right|^2\right)\mathrm{d}x&\leq  2M_1+2\left(c_7+1\right)\left(2c_5^2+1\right)\int_\Omega\left(\left|\nabla p_\varepsilon\right|^2+\left|\nabla n_\varepsilon\right|^2\right)\mathrm{d}x\\
    &+2c_6^2\left(c_7+1\right)\left(c_5^2+1\right).
\end{split}\]
Hence, from \ref{con3} of Assumption \ref{assumption}, there exists a $\tau_0>0$, such that if 
\[\|\nabla \phi_D\|_{L^{\widetilde{\sigma}}(\Omega)}\leq \tau_0, \quad \|\mathbf{v}\|_{H^2(\Omega)}\leq\tau_0,\]
such that 
\[2\left(c_7+1\right)\left(2c_5^2+1\right)<\frac{1}{2},\]
then it follows that  
\begin{equation}\label{E_aprioriest_pnphi}
\begin{split}
   \int_\Omega\left(\left|\nabla p_\varepsilon\right|^2+\left|\nabla n_\varepsilon\right|^2\right)\mathrm{d}x&\leq 4M_1+4c_6^2\left(c_7+1\right)\left(c_5^2+1\right), \\ 
   \int_\Omega\left((p_\varepsilon+n_\varepsilon)\left|\nabla\phi_\varepsilon\right|^2\right)\mathrm{d}x&\leq2M_1+2c_6^2\left(c_7+1\right)\left(c_5^2+1\right)+\frac{c_5^2}{2}+c_6. 
\end{split} \end{equation}
In the remainder of the proof, we shall prove that $p_\varepsilon$ and $n_\varepsilon$ are bounded from above by a constant which is independently of $\varepsilon$, i.e. estimate \eqref{aprio_est2}. We recall $k_0$ given by \eqref{k0} and define 
\[A_k=A_{k,\varepsilon}=\left\{x\in\Omega:~p_\varepsilon(x)>k\right\},\quad 
B_k=B_{k,\varepsilon}=\left\{x\in\Omega:~n_\varepsilon(x)>k\right\},\]
\[\omega(k)=\omega_\varepsilon(k)=\mathrm{meas}~A_k+\mathrm{meas}~B_k\]
for any $k\geq k_0$. Then the function $\omega$ is non-increasing on $[k_0, \infty)$.

By using \eqref{Sebd}, 
\begin{equation}\label{E_lowerbd}
    \begin{split}
        (h-k)\omega(h)\leq& (\mathrm{meas}~A_k)^{1/\sigma}\left(\int_\Omega\left((p_\varepsilon-k)^+\right)^{\widetilde{\sigma}}\mathrm{d}x\right)^{1/\widetilde{\sigma}}\\
        +&(\mathrm{meas}~B_k)^{1/\sigma}\left(\int_\Omega\left((n_\varepsilon-k)^+\right)^{\widetilde{\sigma}}\mathrm{d}x\right)^{1/\widetilde{\sigma}}\\
        \leq& c_0(\omega(k))^{1/\sigma}\left(\int_\Omega\left(\left|\nabla(p_\varepsilon-k)^+\right|^2+\left|\nabla(n_\varepsilon-k)^+\right|^2\right)\mathrm{d}x\right)^{1/2},
    \end{split}
\end{equation}
for all $h>k\geq k_0$, where $1<\sigma<\infty$ if $\dim\Omega=2$, $\sigma=6/5$ if $\dim\Omega=3$ ($\widetilde{\sigma}=\sigma/(\sigma-1)$).

On the other hand, inserting $\varphi_p=(p_\varepsilon-k)^+/\epsilon_p$ into \eqref{E_approx_SS_wk_p} and $\varphi_n=(n_\varepsilon-k)^+/\epsilon_n$ into \eqref{E_approx_SS_wk_n} ($k\geq k_0$) gives
\begin{equation}\label{E_upperbd}
    \begin{split}
        &\int_\Omega\left(\left|\nabla(p_\varepsilon-k)^+\right|^2+\left|\nabla(n_\varepsilon-k)^+\right|^2\right)\mathrm{d}x\\
        \leq& \int_\Omega\left(n_\varepsilon\nabla\phi_\varepsilon\cdot\nabla(n_\varepsilon-k)^+-p_\varepsilon\nabla\phi_\varepsilon\cdot\nabla(p_\varepsilon-k)^+\right)\mathrm{d}x\\
        +&\int_\Omega\left(\frac{p_\varepsilon}{\epsilon_p}\mathbf{v}\cdot\nabla(p_\varepsilon-k)^++\frac{n_\varepsilon}{\epsilon_n}\mathbf{v}\cdot\nabla(n_\varepsilon-k)^+\right)\mathrm{d}x\\
        +&(\alpha_1+\eta_0)\epsilon_n\int_\Omega\left(p_\varepsilon+n_\varepsilon\right)\left|\nabla\phi_\varepsilon\right|\left((p_\varepsilon-k)^++(n_\varepsilon-k)^+\right)\mathrm{d}x\\
        =:&L_1+L_2+L_3.
    \end{split}
\end{equation}
For $\dim\Omega=2$, we fix any $1<t<\min\{2r/(1+r),4/3\}$ ($r>2$), and then fix $q$ sufficiently large such that $0<t/(2-t/q-t)<r$. If $\dim\Omega=3$, one sets $q=3$. Hence, in both cases,
\[\frac{qt}{2q-t-qt}<r,\quad\text{with}\quad r>\dim\Omega.\]
Thus, by \eqref{E_approx_SS_wk_phi}, we have 
\[\begin{split}
    L_1&=\frac{1}{2}\int_{A_k}\left(\nabla\phi_\varepsilon\cdot\nabla n_\varepsilon^2\right)\mathrm{d}x-\frac{1}{2}\int_{B_k}\left(\nabla\phi_\varepsilon\cdot\nabla p_\varepsilon^2\right)\mathrm{d}x\\
    &=\frac{1}{2}\int_\Omega\left((p_\varepsilon-n_\varepsilon)\left[\left(n_\varepsilon^2-k^2\right)^+-\left(p_\varepsilon^2-k^2\right)^+\right]\right)\mathrm{d}x\\
    &\leq 0.
\end{split}\]
We recall that the velocity $\mathbf{v}$ satisfies 
\[\nabla\cdot\mathbf{v}=0\ \text{in}\ \Omega,\quad \mathbf{v}\cdot\nu=0\ \text{on}\ \partial\Omega,\]
so using integration by parts, we have 
\[\begin{split}
    L_2&= \frac{1}{2\epsilon_p}\int_{A_k}\left(\mathbf{v}\cdot\nabla p_\varepsilon^2\right)\mathrm{d}x+\frac{1}{2\epsilon_n}\int_{B_k}\left(\mathbf{v}\cdot\nabla n_\varepsilon^2\right)\mathrm{d}x\\
    &+\frac{k}{2\epsilon_p}\int_\Omega\left(\mathbf{v}\cdot\nabla(p_\varepsilon-k)^+\right)\mathrm{d}x+\frac{k}{2\epsilon_n}\int_\Omega\left(\mathbf{v}\cdot\nabla(n_\varepsilon-k)^+\right)\mathrm{d}x\\
    &\leq 0.
\end{split}\]
Next, we have
\[\begin{split}
    L_3&\leq \frac14\int_\Omega\left(\left|\nabla(p_\varepsilon-k)^+\right|^2+\left|\nabla(n_\varepsilon-k)^+\right|^2\right)\mathrm{d}x\\
    &+{2c_0^2\kappa_0^2}\int_\Omega\left(p_\varepsilon+n_\varepsilon\right)\left|\nabla\phi_\varepsilon\right|^2\mathrm{d}x\left(\int_\Omega\left(p_\varepsilon+n_\varepsilon\right)^6\mathrm{d}x\right)^{1/6}\left(\omega(k)\right)^{1/2}.
\end{split}\]
Hence, \eqref{E_upperbd} becomes 
\begin{equation}\label{E_upperbd_2}
    \begin{split}&\int_\Omega\left(\left|\nabla(p_\varepsilon-k)^+\right|^2+\left|\nabla(n_\varepsilon-k)^+\right|^2\right)\mathrm{d}x\\
        \leq&{2c_0^2\kappa_0^2}\int_\Omega\left(p_\varepsilon+n_\varepsilon\right)\left|\nabla\phi_\varepsilon\right|^2\mathrm{d}x\left(\int_\Omega\left(p_\varepsilon+n_\varepsilon\right)^6\mathrm{d}x\right)^{1/6}\left(\omega(k)\right)^{1/2}\\
        \leq& {4c_0^2\kappa_0^2}\left(2c_5^2\left(M_1+c_6^2\right)+c_6\right)\left(M_1+c_6^2+c_0^2\right)\left(\omega(k)\right)^{1/2}\end{split}
\end{equation}
 \eqref{E_aprioriest_pnphi}, \eqref{E_lowerbd} and \eqref{E_upperbd_2} infer that
\[(h-k)\omega(h)\leq {2c_0^2\kappa_0}\sqrt{\left(2c_5^2\left(M_1+c_6^2\right)+c_6\right)\left(M_1+c_6^2+c_0^2\right)}(\omega(k))^{1/\sigma+1/4}\]
for all $h>k\geq k_0$. 

As $1/t+1/4>1$, then an iteration argument (Lemma 4.1 from \cite{stampacchia1965probleme}) infers that $\omega(k_0+M_2)=0$, where $M_2=c(\omega(k_0))^{1/\sigma-3/4}2^{(1/\sigma+1/4)/(1/\sigma-3/4)}$ with constant $c>0$, i.e.
\begin{equation}\label{E_upperbd_pn}
    p_\varepsilon\leq k_0+M_2,\quad n_\varepsilon\leq k_0+M_2,\quad\text{a.e. in}~\Omega,~\forall~\varepsilon>0.
\end{equation}
Setting 
\[C=\max\left\{M_2,~4M_1+4c_6^2, ~ 2c_5^2\left(M_1+c_6^2\right)+c_6\right\}\]
\eqref{E_aprioriest_pnphi} and \eqref{E_upperbd_pn} imply \eqref{aprio_est}.

This concludes the proof of Theorem \ref{apriori_est_thm}.
\end{proof}				
\subsection{Passage to the Limit}\label{pass2limit}
	By passing to a subsequence, from \eqref{aprio_est}, we conclude that \[\begin{split}
    p_\varepsilon\to p,~ n_\varepsilon\to n,~\phi_\varepsilon\to\phi,~\text{weakly in}~ H^1(\Omega),\\ \text{strongly in}~L^2(\Omega),
    \text{and a. e. in}~\Omega,~\text{as}~\varepsilon\to 0.\end{split}\]		
    In particular, the bounds on $p_\varepsilon$ and $n_\varepsilon$ in \eqref{aprio_est2} continue to hold after letting $\varepsilon\to 0$ therein, i.e., \eqref{SS_pos}. Analogously, \eqref{E_approx_SS_bdy} implies \eqref{SS_bdy}.			
From \eqref{E_approx_SS_wk_phi}, we infer that $\nabla\phi_\varepsilon\to\nabla\phi$ strongly in $L^2(\Omega)$ as $\varepsilon\to 0$. Moreover, \eqref{E_approx_SS_wk_p} and \eqref{E_approx_SS_wk_n} imply $\nabla p_\varepsilon\to\nabla p$ and $\nabla n_\varepsilon\to\nabla n$ strongly in $L^2(\Omega)$ as $\varepsilon\to 0$. Then the passage to limit $\varepsilon\to0$ in \eqref{E_approx_SS_wk_phi}-\eqref{E_approx_SS_wk_n} leads to \eqref{SS_wk_phi}-\eqref{SS_wk_n}.

This concludes the proof of Theorem \ref{SS_Existence}.

\section{Numerical Simulations} \label{numerical}
This section presents a comprehensive numerical assessment of the shifted logarithmic SUPG--FEM solver developed in this work. We evaluate its accuracy, robustness, and geometric versatility through a systematic suite of numerical experiments, organized according to the dimensionality of the computational domain: two‑dimensional planar problems and three‑dimensional axisymmetric problems. The verification campaign includes manufactured solution convergence tests, parameter stress tests, and comparisons with analytical solutions, thereby validating both the theoretical existence result Theorem \ref{SS_Existence} and the practical performance of the solver across a wide range of operating regimes.

Section \ref{sec3.1two} assesses the solver on two‑dimensional domains: a unit square for convergence verification, an L‑shaped corner for singularity handling, and an annular corona geometry for curved boundaries and field concentration. Section \ref{sec3.2three} extends the method to axisymmetric problems, validating it on a three‑dimensional domed cylinder that mimics a MEMS capacitive switch, thereby demonstrating the solver’s capability for realistic device geometries.

\subsection{Two Dimensional Tests} \label{sec3.1two}
Three representative two dimensional configurations are examined: a manufactured solution on the unit square to verify asymptotic convergence rates; an L shaped domain with a re entrant corner; and an annular domain that mimics the classical coaxial corona discharge geometry. The latter two cases are chosen because their geometric features are ubiquitous in realistic MEMS structures and discharge devices.\\

\noindent\textbf{Unit Square}

 This test purpose is to verify that the finite element assembly, the exponential nonlinearities, and the SUPG stabilisation are implemented without algebraic or consistency errors. In practical MEMS simulations, such code verification is a prerequisite before tackling complex geometries and nonlinear discharge phenomena.

The method of manufactured solutions (MMS) \cite{Roache1998} is employed on the unit square $\Omega = (0,1)^2$ to isolate the discretisation error of the piecewise-linear ($P_1$) Galerkin approximation. A smooth potential $\phi(x,y) = \sin(\pi x) \sin(\pi y)$ is prescribed, and the source terms together with the Dirichlet data are constructed analytically so that the shifted log formalism reproduces this field exactly. This purely numerical test has no direct physical analogue; its purpose is to verify that the finite element assembly, the exponential nonlinearities, and the SUPG stabilisation are implemented correctly. Table~\ref{table1} reports the $L^2$ and $H^1$ seminorm errors and the estimated order of convergence (EOC) on a sequence of uniformly refined unstructured triangular meshes.

\begin{longtblr}[
  caption = {Convergence of the shifted log SUPG FEM on a manufactured solution \label{table1}},
]{
  width = \linewidth,
  colspec = {Q[44]Q[98]Q[210]Q[181]Q[210]Q[181]},
  cells = {c},
  cell{2}{4} = {font=\bfseries},
  cell{2}{6} = {font=\bfseries},
  hline{1,6} = {-}{0.08em},
  hline{2} = {-}{},
}
N  & h      & $L^2 error$ & $L^2 EOC$ & $H^1 error$ & $H^1 EOC$ \\
8  & 0.125  & 5.91e-2     & –         & 3.71e-1     & –         \\
16 & 0.0625 & 1.32e-2     & 2.16      & 1.54e-1     & 1.27      \\
32 & 0.0312 & 3.02e-3     & 2.12      & 7.36e-2     & 1.07      \\
64 & 0.0156 & 7.19e-4     & 2.07      & 3.64e-2     & 1.01

\end{longtblr}
The observed rates approach the theoretical values of 2.0 ($L^2$) and 1.0 ($H^1$) for continuous $P_1$ elements as the mesh is refined, confirming that both the Poisson sub-solver and the shifted log transformations are implemented without algebraic or consistency errors. The slight deviation from the asymptotic rates on the coarser meshes is expected and vanishes as $h \to 0$.\\

\noindent\textbf{L shaped Domain}

The L-shaped region $[0,2] \times [0,1] \cup [0,1] \times [1,2]$ contains a re-entrant corner at $(1,1)$ with an interior angle of $270^\circ$. Such corners are common in etched cavities, comb drive anchors, and step coverage features of MEMS devices \cite{Senturia2001}. In these locations, the electric field becomes singular, leading to locally enhanced charge injection and dielectric charging — a primary reliability concern for capacitive RF MEMS switches. Accurate simulation of field enhancement is therefore essential for predicting device lifetime.

In electrostatics, the electric field near a re-entrant corner behaves as $|\nabla\phi| \propto r^{\gamma-1}$, where $r$ is the distance from the corner and $\gamma \approx 0.544$ for the L-shape \cite{Grisvard1992}. This singularity causes strong local field enhancement, which in dielectric charging phenomena can lead to locally amplified charge injection \cite{Molinero2006}. The shifted log solver converges on this domain in 19 outer Gummel iterations and yields strictly positive carrier densities with $\min(p) = \min(n) = p_{\min}$ to machine precision. Figure~\ref{fig:lshape}  displays contour plots of $\phi$, $p$, and $n$; the solution is free of spurious oscillations, and the electric field correctly captures the corner singularity without any stabilisation failure. The positivity of the carrier densities is maintained even in the immediate vicinity of the singularity, confirming that the algebraic positivity mechanism is robust with respect to geometric irregularities.\\
\begin{figure}[htbp]
  \centering
  \includegraphics[width=\textwidth]{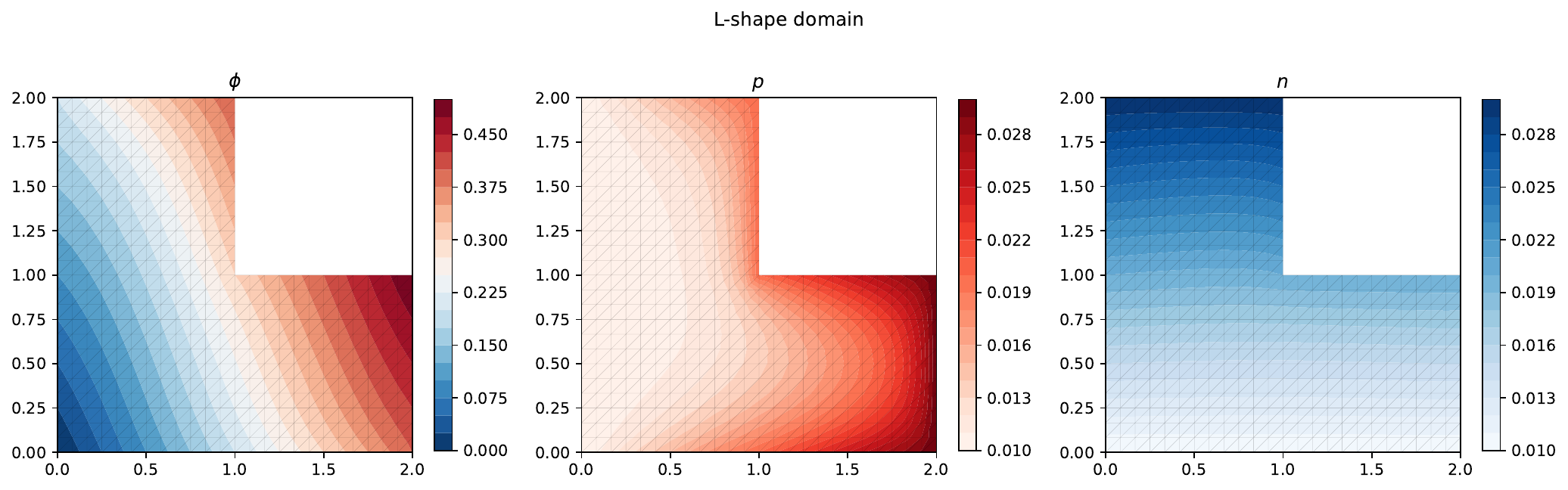}
  \caption{Contour plots of the electric potential $\phi$, positive ion 
           density $p$ and electron density $n$ on the L-shaped 
           domain}
  \label{fig:lshape}
\end{figure}

\noindent\textbf{Annular Domain (Corona Geometry)}

The annulus $r \in [1,5]$, with inner radius $r_{\mathrm{in}} = 1$ and outer radius $r_{\mathrm{out}} = 5$, reproduces the coaxial electrode arrangement analysed by Budd \cite{Budd1991} for positive corona discharges. This geometry is fundamental to many industrial applications: high‑voltage power transmission lines (where corona causes energy loss and electromagnetic interference), electrostatic precipitators (where corona charges dust particles for collection), and corona ignition systems for internal combustion engines.

In such devices, a thin wire (inner electrode) is held at high potential relative to a grounded outer cylinder, producing a strongly non-uniform electric field. Even in the space-charge-free limit, the Laplace solution yields the logarithmic profile $\phi(r) \propto \log(r_{\mathrm{out}} / r) / \log(r_{\mathrm{out}} / r_{\mathrm{in}})$, leading to a radial field gradient $|\nabla\phi| \propto 1/r$. This field concentration near the inner electrode localises the ionisation activity and is responsible for the characteristic corona structure \cite{Raizer1991}. 

The annulus therefore tests the ability of the linear triangular mesh to resolve curved Dirichlet boundaries and strongly inhomogeneous fields. The solver converges in 22 outer Gummel iterations. As shown in Figure~\ref{fig:annulus}, the computed radial potential agrees with the analytic Laplace solution within the resolution of the mesh, and the carrier densities remain strictly positive. Ionisation is correctly confined to the inner region where the field is strongest, consistent with the physical expectation that electron avalanches initiate near the wire. Collectively, the L-shape and annulus results demonstrate that the solver handles both non-convex polygonal and curved domains with identical algorithmic settings, providing oscillation-free solutions and exact algebraic positivity.
\begin{figure}[htbp]
  \centering
  \includegraphics[width=\textwidth]{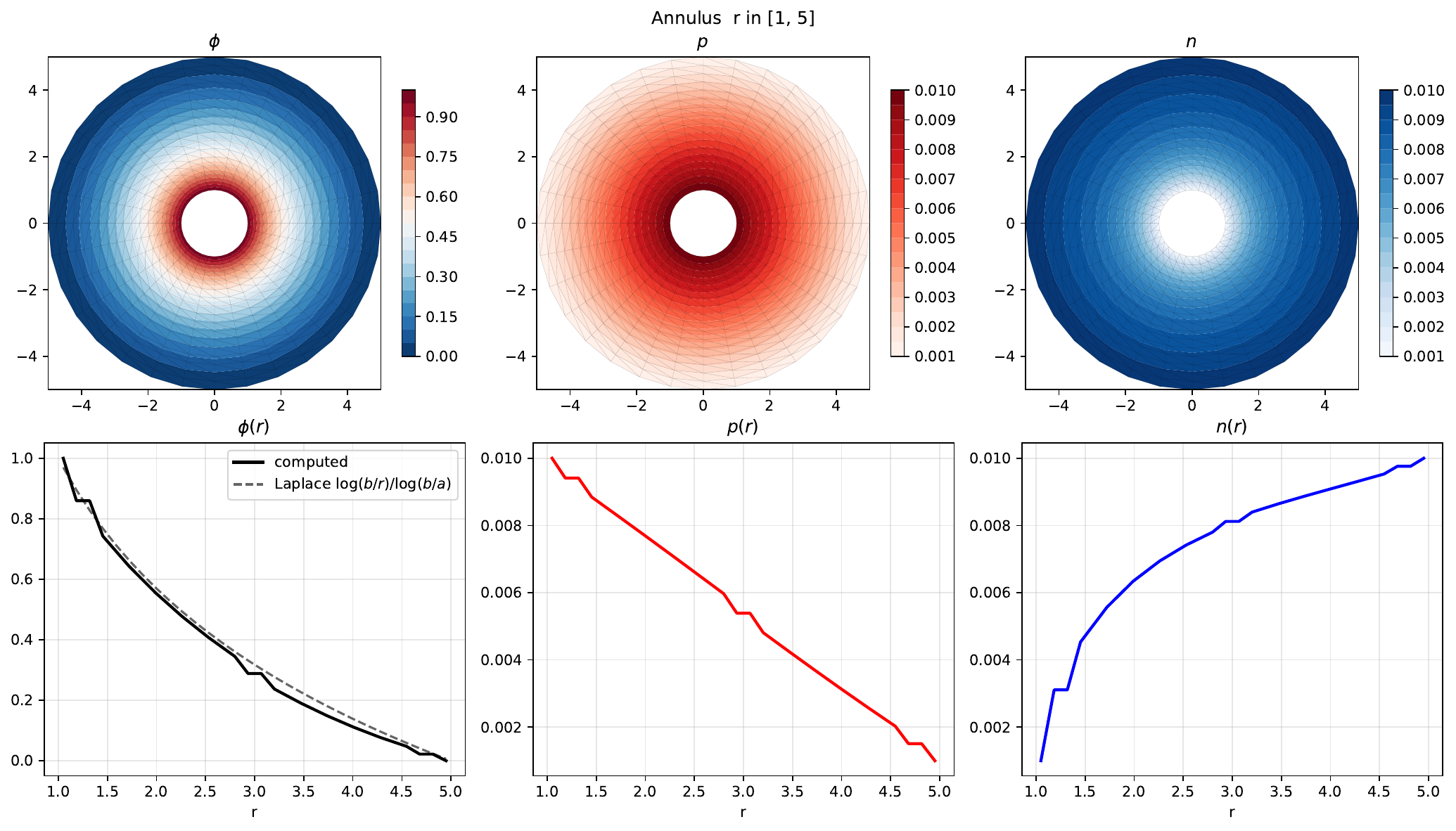}
  \caption{Solutions on the annular domain $r\in[1,5]$. 
           Top row: contour plots of $\phi$, $p$ and $n$. 
           Bottom row: radial profiles showing $\phi(r)$ 
           compared with the analytic Laplace solution 
           $\phi(r) \propto \log(r_{\mathrm{out}} / r) / \log(r_{\mathrm{out}} / r_{\mathrm{in}})$}
  \label{fig:annulus}
\end{figure}

\subsection{Axisymmetric Corona Test} \label{sec3.2three}

Many discharge configurations of practical relevance, such as needle‑to‑plane gaps, coaxial cylinders, and curved‑electrode coronas, possessing rotational symmetry. Exploiting axisymmetry reduces a full 3D simulation to a 2D problem on the meridian section, drastically cutting computational cost while preserving essential physical features. This makes axisymmetric modelling indispensable for parametric studies and iterative design of MEMS devices.

To extend the solver to these geometries without incurring the computational cost of a full three-dimensional mesh, we exploit axisymmetry. Under the assumption that all fields depend only on the radial coordinate $r = \sqrt{x_1^2 + x_2^2}$ and the axial coordinate $y$, the three-dimensional weak form reduces, after angular integration, to a two-dimensional problem on the meridian section $\Omega_m$:
\begin{equation}
\int_{\Omega_m} (\mathcal{F}_{rd} \cdot \nabla_{rd} w) \, r \, dr \, dy = \int_{\Omega_m} S \, w \, r \, dr \, dy,
\end{equation}
where $\nabla_{rd} = (\partial_r, \partial_y)$. The sole algorithmic modification with respect to the planar solver is the multiplication of all element integrals by the cylindrical Jacobian $r$. The symmetry axis $r=0$ is handled automatically as a natural Neumann boundary \cite{Zienkiewicz2013}. All other features---shifted log positivity, consistent SUPG stabilisation, Gummel--Newton iteration, and vectorised assembly---are inherited directly from the two-dimensional solver.

As a demanding verification, we consider the three-dimensional domain
\begin{equation}
\Omega = \{(x_1, x_2, y) : 0 \le y \le (x_1^2 + x_2^2)^{2/3} + 2, \; x_1^2 + x_2^2 \le 25\},
\end{equation}
whose boundary consists of a curved upper dome $A$: $y = r^{4/3} + 2$ ($r \le 5$), a flat circular floor $B$: $y = 0$ ($r \le 5$), and a cylindrical side wall $C$: $r = 5$ ($0 < y < 10.55$). This geometry is inspired by a point-to-plane corona with an energised curved electrode \cite{Budd1991, Morrow1997}. The dome is held at a fixed potential $V_0$ and supplies ions ($p_D = 10^{-2}$). The floor and the side wall are grounded ($\phi = 0$) and carry a weak background carrier reservoir ($p_D = n_D = 10^{-3}$). No external advection is present ($\mathbf{v} = \mathbf{0}$). The physical parameters are $\epsilon_\phi = \epsilon_p = \epsilon_n = 0.05$, $\alpha_1 = 1.5$, $\alpha_2 = 0.4$, $\eta_0 = 0.1$, with $p_{\min} = n_{\min} = 10^{-5}$ and $\delta = 10^{-5}$.

Figure~\ref{fig:domsol} shows the converged solution at $V_0 = 1.0$ in the meridian section. The potential decays smoothly from the dome toward the floor and side wall. The positive ion density is sharply concentrated near the dome, where ions are injected, and decreases monotonically into the bulk. The electron density develops a weak interior maximum in a region of elevated electric field, where impact ionisation is most active. 

Figure~\ref{fig:domfie} displays the electric field magnitude; the strongest field occurs at the rim corner $(r=5, y\approx 10.55)$, where the energised dome meets the grounded cylinder---a re-entrant edge with a Dirichlet discontinuity.
\begin{figure}[htbp]
  \centering
  \includegraphics[width=\textwidth]{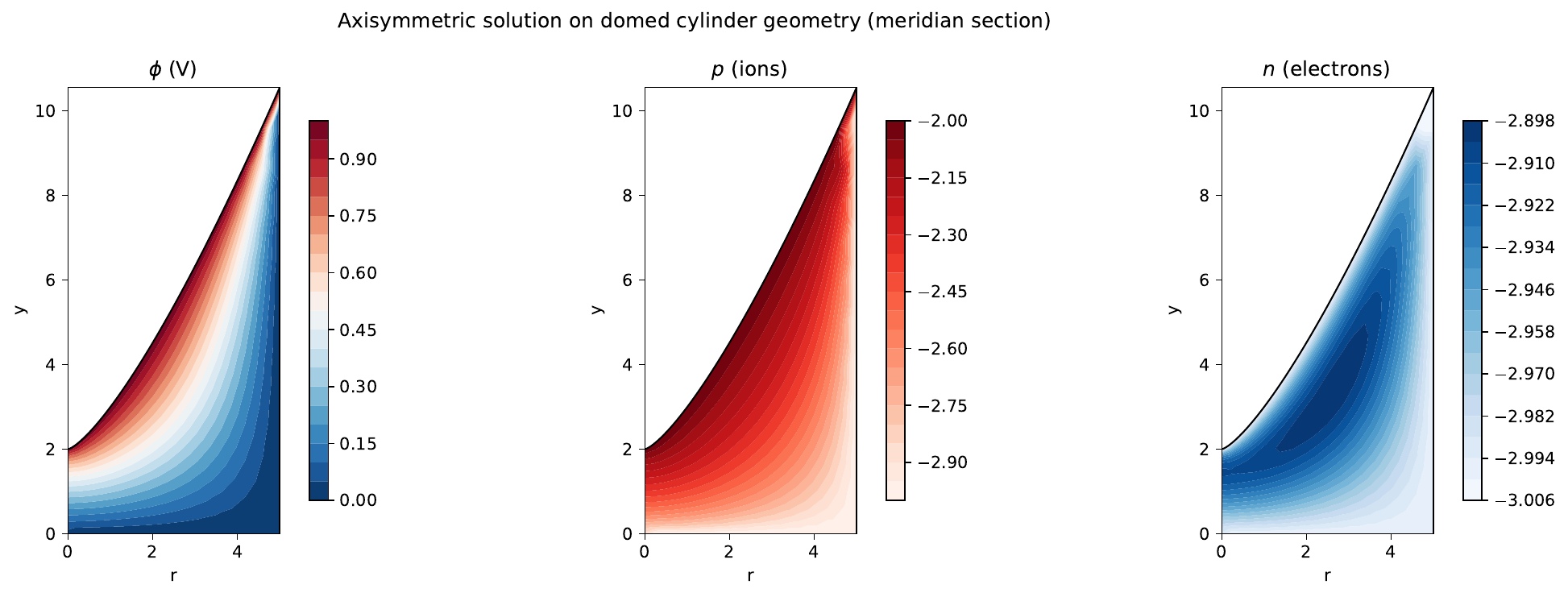}
  \caption{Converged solution at $V_0 = 1$ in the meridian section: $\phi$ (left), $\log_{10} p$ (middle), $\log_{10} n$ (right). Black lines show the boundary of the domain}
  \label{fig:domsol}
\end{figure}
\begin{figure}
  \centering
  \includegraphics[scale=0.35]{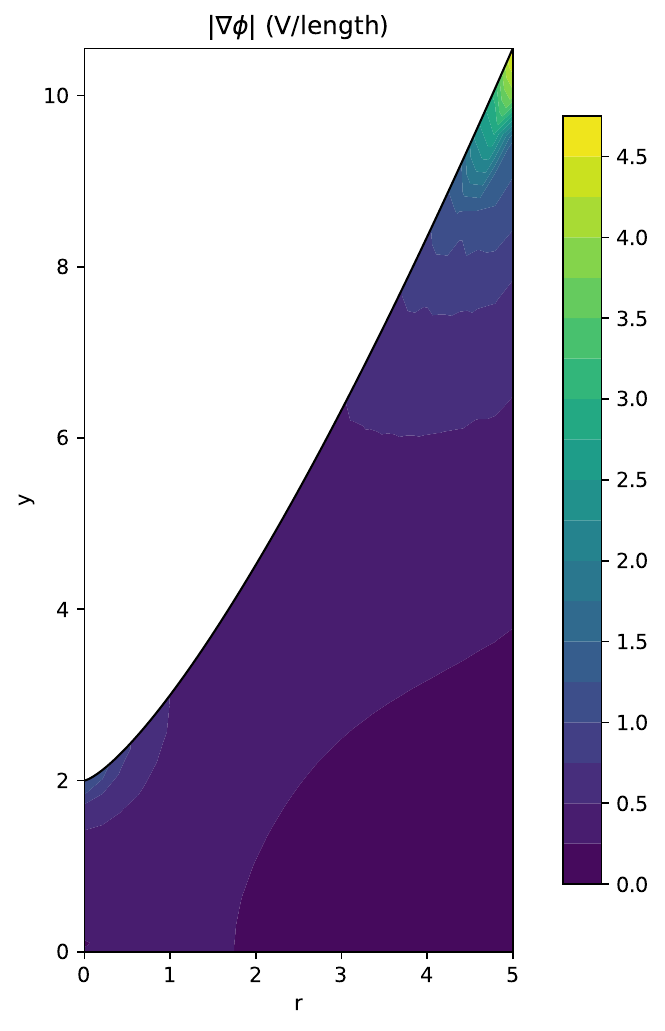}
  \caption{$|\nabla \phi|$ in the meridian section. Maximum is at the rim corner $(r=5, y=10.55)$ due to the Dirichlet jump between $A$ and $C$}
  \label{fig:domfie}
\end{figure}

A secondary, weaker field channel connects the dome apex to the grounded floor along the symmetry axis, consistent with the expected Laplacian field structure. A three-dimensional rendering of the solution for $V_0=1.5$, obtained by revolving the meridian section, is provided in Figure~\ref{fig:3d}.

The rim where the dome meets the cylinder constitutes a re‑entrant edge with a Dirichlet discontinuity, analogous to the corner singularity in the L‑shaped domain but in three dimensions. As expected from elliptic theory, the electric field exhibits a logarithmic singularity at this edge. The solver stably captures this feature, demonstrating its ability to handle practical 3D geometries with complex boundary conditions.
\begin{figure}[htbp]
  \centering
  \includegraphics[width=\textwidth]{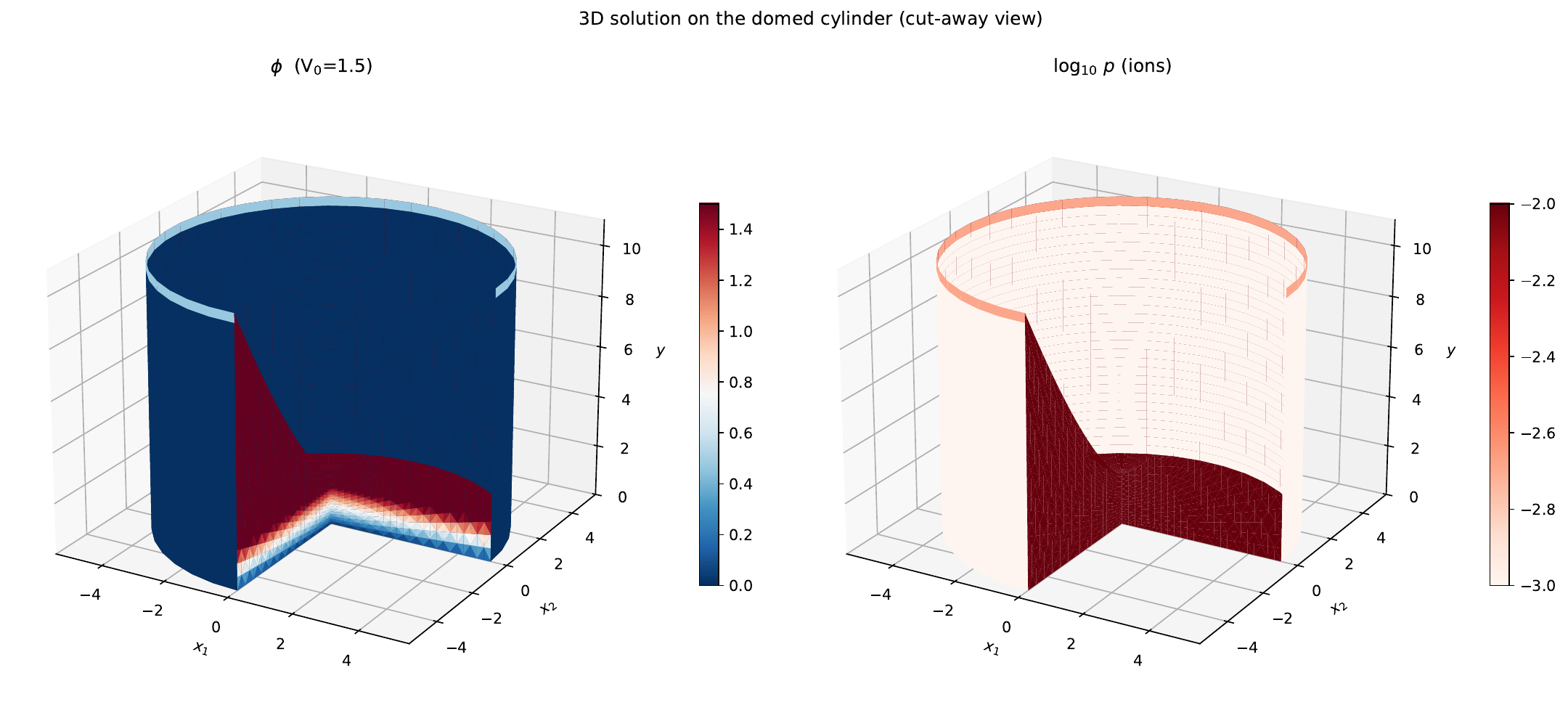}
  \caption{Three-dimensional rendering of the converged solution at $V_0 = 1.5$ by revolution of the meridian section. Left: electric potential $\phi$. Right: logarithm of positive ion density $\log_{10} p$}
  \label{fig:3d}
\end{figure}

To examine the response of the discharge to the applied voltage, we sweep $V_0 \in \{0.5, 1.0, 1.5, 2.0\}$ on a fixed meridian mesh with $N_r = 24$, $N_y = 20$ (525 nodes, 960 triangles). Profiles along the symmetry axis are shown in Figure~\ref{fig:domsw}. Quantitative results are summarised in Table~\ref{table2}. 

The number of outer Gummel iterations remains essentially constant (28--29), demonstrating that the nonlinear solver is robust with respect to voltage variation. The maximum positive ion density stays pinned at the Dirichlet value $10^{-2}$; the maximum electron density and the peak field magnitude both grow approximately linearly with $V_0$. The linear scaling of the field indicates that space charge feedback is weak under the chosen parameters, i.e., the system operates in a Laplace-dominated regime. This behaviour is physically plausible for a low ionisation scenario, where the carrier densities are insufficient to significantly perturb the applied Laplacian field \cite{Raizer1991}.

From a design perspective, such linear response allows straightforward extrapolation to higher voltages until space‑charge effects eventually become significant, at which point the device would transition toward a saturation or breakdown regime.

\begin{figure}[htbp]
  \centering
  \includegraphics[scale=0.46]{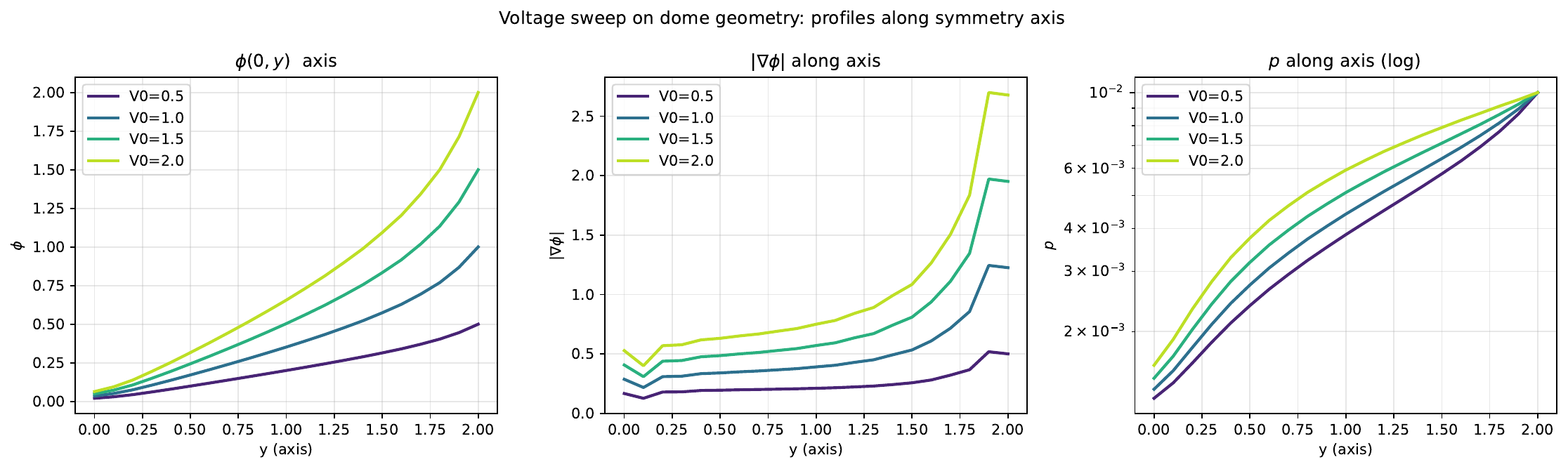}
  \caption{Voltage sweep $V_0 \in \{0.5, 1.0, 1.5, 2.0\}$. Left: electric potential $\phi$ along the symmetry axis. Middle: electric field magnitude $|\nabla \phi|$ along the axis. Right: ion density $p$ along the axis}
  \label{fig:domsw}
\end{figure}

\begin{longtblr}[
  caption = {Voltage sweep on the domed cylinder problem \label{table2}},
]{
  width = \linewidth,
  colspec = {Q[92]Q[92]Q[252]Q[260]Q[240]},
  cells = {c},
  hline{1,6} = {-}{0.08em},
  hline{2} = {-}{},
}
$V_0$ & Iterations & $max ~p$               & $max ~n$                & $\max \|\nabla \phi\|$ \\
0.5   & 29         & $1.00 \times 10^{-2}$ & $1.101 \times 10^{-3}$ & 2.329                  \\
1.0   & 28         & $1.00 \times 10^{-2}$ & $1.262 \times 10^{-3}$ & 4.659                  \\
1.5   & 28         & $1.00 \times 10^{-2}$ & $1.585 \times 10^{-3}$ & 6.988                  \\
2.0   & 28         & $1.00 \times 10^{-2}$ & $2.126 \times 10^{-3}$ & 9.318                  
\end{longtblr}

\ 

To assess spatial convergence, we evaluate the total integrated ion charge $I_p = \int_{\Omega} p \, dV = 2\pi \int_{\Omega_m} p \, r \, dr \, dy$ at $V_0 = 0.5$ on four successively refined meshes. The results are listed in Table \ref{table3}. The integrated charge converges to four significant figures, and the relative error approximately halves with each refinement. This rate is consistent with an $O(h)$ convergence for an $L^1$-type integrated quantity in the presence of a corner singularity \cite{Grisvard1992}. The maximum electron density remains stable to four significant digits across all meshes, and the number of outer iterations is exactly 29, confirming mesh independent convergence of the Gummel--Newton scheme.

\ 

\begin{longtblr}[
  caption = {Mesh refinement study for the domed cylinder at $V_0 = 0.5$ \label{table3}},
]{
  width = \linewidth,
  colspec = {Q[163]Q[90]Q[90]Q[240]Q[117]Q[233]},
  column{odd} = {c},
  column{2} = {c},
  column{4} = {c},
  hline{1,6} = {-}{0.08em},
  hline{2} = {-}{},
}
$N_r \times N_y$ & Nodes  & Elements & $max ~n$                 & $I_p$     & Relative error         \\
$12\times 10$    & $143$  & $240$    & $1.1010 \times 10^{-3}$ & $2.13476$ & $1.696 \times 10^{-3}$ \\
$16\times 14$    & $255$  & $448$    & $1.1009 \times 10^{-3}$ & $2.13350$ & $1.103 \times 10^{-3}$ \\
$24\times 20$    & $525$  & $960$    & $1.1006 \times 10^{-3}$ & $2.13222$ & $5.007 \times 10^{-4}$ \\
$36\times 30$    & $1147$ & $2160$   & $1.1006 \times 10^{-3}$ & $2.13115$ & 1.1006e-03           
\end{longtblr}

This study is particularly important for practical MEMS simulations: although local field values at geometric singularities may not converge pointwise, integral quantities such as total dielectric charge — which directly correlate with device degradation and lifetime — converge reliably. The solver thus provides a trustworthy tool for quantitative reliability assessment.

The rim where the dome meets the cylinder constitutes a re-entrant edge with a Dirichlet discontinuity. As is well known from elliptic theory, the exact electric field possesses a logarithmic singularity at this corner \cite{Grisvard1992}. Numerically, the maximum field magnitude grows upon mesh refinement as the singularity is better resolved, whereas global quantities such as the total charge remain finite and converge. This behaviour is a genuine geometric feature, not a numerical artefact, and it does not compromise the overall fidelity of the simulation. In engineering practice, the singularity can be regularised by rounding the corner or by imposing a smooth transition in the Dirichlet data; the present results confirm that the solver handles the unmodified geometry in a stable manner.

The axisymmetric extension preserves all the desirable properties of the planar solver: strict positivity, robust convergence under parameter variation, and mesh independent iteration counts. It furthermore demonstrates that the shifted log SUPG–FEM framework can efficiently simulate fully three dimensional discharge structures with rotational symmetry, providing a solid foundation for future extensions to general 3-D geometries and transient phenomena.

\appendix
\counterwithin*{equation}{section}
\renewcommand\theequation{\thesection\arabic{equation}}
\section{Gummel‑Newton Solver} \label{app2}
In this appendix, we present a comprehensive numerical validation of the theoretical existence result stated in Theorem \ref{SS_Existence} and illustrate the performance of a robust finite element solver for the stationary drift‑diffusion model \eqref{S_EPSys}. We adopt the shifted logarithmic transformation of the carrier densities to enforce strict positivity by construction and employ a consistent SUPG stabilization to handle convection‑dominated regimes. All computations are performed on unstructured triangular meshes using piecewise linear ($P_1$) continuous finite elements.
 
\subsection{Shifted‑Log Transformation and FEM Discretization} \label{sec3.1}
To guarantee strict positivity of the carrier densities throughout the nonlinear solution process, a property that is essential for physical relevance and in accordance with the a priori bounds (e.g. \eqref{aprio_est}).
We introduce the auxiliary variables
\begin{equation}
P = \log(p - p_{\min}), \quad N = \log(n - n_{\min}),
\end{equation}
with small positive floors $p_{\min} = n_{\min} = 10^{-6}$.
The electric potential $\phi$ is retained in its primitive form. Substituting
\begin{equation}
    p = p_{\min} + e^P, \quad n = n_{\min} + e^N,
\end{equation} 
into the governing equations \eqref{S_EPSys} yields the transformed system
\begin{subequations}
\begin{equation}
		-\Delta \phi=\frac{(p_{\min} + e^P)-(n_{\min} + e^N)}{\epsilon_\phi},\quad x\in\Omega,
	\end{equation}
	\begin{equation}
		-\nabla\cdot \left(\epsilon_p e^P \nabla P + (p_{\min} + e^P)\left(\epsilon_p\nabla \phi-\mathbf{v}\right)\right)=\widetilde{F}_{\delta} (N, \phi),\quad x\in\Omega,
	\end{equation}
	\begin{equation}
		-\nabla\cdot \left(\epsilon_n e^N \nabla N-(n_{\min} + e^N)\left(\epsilon_n\nabla \phi+\mathbf{v}\right)\right)=\widetilde{F}_{\delta} (N, \phi),\quad x\in\Omega,
	\end{equation}
\end{subequations}
where the regularised ionisation source is
\begin{equation}
    \widetilde{F}_{\delta} (N, \phi) = \epsilon_n (n_{\min} + e^N) E_{\delta}(\phi) (\alpha_1 e^{-\frac{\alpha_2}{E_{\delta}(\phi)}}-\eta_0)
\end{equation}
with $E_{\delta}(\phi) = \sqrt{|\nabla \phi|^2+\delta^2}$ and a regularisation parameter $\delta = 10^{-6}$ in the numerical experiments.
The transformed transport equations are nonlinear in P and N but enjoy locally Lipschitz continuous nonlinearities, making them amenable to Newton’s method.

The weak formulation is discretised using standard $P_1$ Galerkin finite elements on a triangulation of the domain $\Omega$. For the Poisson equation the Galerkin approximation is used without modification. For the carrier transport equations, the standard Galerkin method is augmented with a consistent SUPG stabilisation to suppress spurious oscillations in convection dominated regimes. Specifically, for the ion transport equation, the test function $\psi_h$ is replaced by
\begin{equation}
    \psi_h \rightarrow \psi_h + \tau_k (w_p \cdot \nabla \psi_h)
\end{equation}
where $w_p = \epsilon_p \nabla \phi - \mathbf{v}$ is the advective velocity for ions, $h_k$ denotes the element length of triangle $K$ in the direction of
$w_p$ and $\tau_k$ is an element wise stabilisation parameter defined by
\begin{equation}
    \tau_k = \frac{h_k}{2|w_p|} \min \left\{\frac{Pe_k}{3}, 1\right\}, \quad Pe_k = \frac{|w_p| h_k}{2 \epsilon_p}
    \label{tau_def}
\end{equation}
An analogous definition holds for the electron equation with velocity $w_n = -(\epsilon_n \nabla \phi + \mathbf{v})$. The SUPG terms are incorporated consistently into both the residual vector and the Jacobian matrix to ensure that the fixed point of the Newton iteration corresponds to the stabilised weak solution.

\subsection{Gummel‑Newton Iteration} 

The discrete nonlinear system arising from the finite element discretisation of \ref{sec3.1} is solved by a decoupled Gummel‑type outer iteration with damped Newton inner solves. This strategy exploits the natural asymmetry of the coupling: the Poisson equation determines $\phi$ given the space charge, while the transport equations for $N$ and $P$ are driven primarily by the electric field. Decoupling reduces the size of the linear systems to be solved at each nonlinear step and allows tailored preconditioning if iterative linear solvers are employed. 

Let $\left(\phi^{k}, P^{k}, N^{k}\right)$ denote the approximation at the beginning of the k-th outer iteration. One Gummel cycle consists of the three sequential steps described below, which produce intermediate (undamped) quantities denoted by a tilde: $\tilde{\phi}$, $\tilde{P}$, $\tilde{N}$.

\subsubsection{Poisson Step}
With $\mathbf{P}=\mathbf{P}^{k}$ and $\mathbf{N}=\mathbf{N}^{k}$ frozen, the Poisson equation becomes linear in $\phi$. Its $P_{1}$ finite element discretisation on the triangulation $\mathcal{T}_{h}$leads to the linear algebraic system
\begin{equation}
    \mathbf{A}_{\phi} \tilde{\bm{\phi}} = \mathbf{f}_{\phi}
\label{poisson}
\end{equation}
where $\mathbf{A}_{\phi}\in\mathbb{R}^{N_n\times N_n}$ is the standard stiffness matrix and $\mathbf{f}_{\phi}\in\mathbb{R}^{N_n}$ is the right-hand-side vector. The assembly is performed element by element. On a triangle $K$ with vertices $\{\bm{x}_{a}\}_{a=1}^{3}$, the local stiffness matrix $\bm{A}_{K}^{\phi}\in\mathbb{R}^{3\times3}$ has entries. Let $\lambda_a^K$, $a=1,2,3$, denote the piecewise linear nodal basis 
functions (barycentric coordinates) on triangle $K$. Their gradients 
$\nabla\lambda_a^K$ are constant on $K$.
\begin{equation}
(\bm{A}_{K}^{\phi})_{ab}=|K|\,(\nabla\lambda_{b}^{K})\cdot(\nabla\lambda_{a}^{K}),\quad a,b=1,2,3,
\end{equation}
where $|K|$ is the area of $K$. The local right-hand-side vector $\mathbf{f}_K^\phi \in \mathbb{R}^3$ is assembled using a three point Gaussian quadrature rule with weights $\{w_q\}_{q=1}^3$ and barycentric evaluation points $\{\bm{\xi}_q\}_{q=1}^3$:
\begin{equation}
(\mathbf{f}_K^\phi)_a = \sum_{q=1}^{3} w_q \, |K| \, \frac{1}{\epsilon_\phi} \left[ (p_{\min} + e^{P_h(\bm{x}_q)}) - (n_{\min} + e^{N_h(\bm{x}_q)}) \right] \xi_{q,a},
\end{equation}
with the interpolated values $P_h(\bm{x}_q) = \sum_{c=1}^{3} P_c^k \, \xi_{q,c}$ and $N_h(\bm{x}_q) = \sum_{c=1}^{3} N_c^k \, \xi_{q,c}$.

Dirichlet boundary conditions $\phi = \phi_D$ on $\Gamma_D$ are enforced by modifying the linear system. Let $\mathcal{I}_{\mathrm{free}}$ and $\mathcal{I}_{\mathrm{Dir}}$ denote the index sets of free and Dirichlet nodes, respectively. For each Dirichlet node $i \in \mathcal{I}_{\mathrm{Dir}}$:
\begin{enumerate}
    \item $(A_\phi)_{ii} \gets 1$, and $(A_\phi)_{ij} \gets 0$ for all $j \neq i$;
    \item $(A_\phi)_{ji} \gets 0$ for all $j \neq i$;
    \item $(\mathbf{f}_\phi)_i \gets \phi_D(\bm{x}_i)$;
    \item For each free node $j \in \mathcal{I}_{\mathrm{free}}$ adjacent to $i$, subtract the contribution: $(\mathbf{f}_\phi)_j \gets (\mathbf{f}_\phi)_j - (A_\phi)_{ji} \, \phi_D(\bm{x}_i)$.
\end{enumerate}
The reduced symmetric positive definite system for the free nodes is solved by a sparse direct solver, yielding $\widetilde{\bm{\phi}}_{\mathrm{free}}$. The full intermediate potential vector $\widetilde{\bm{\phi}}$ is then constructed by assigning the computed free values and the prescribed Dirichlet values to their respective positions.

After obtaining $\widetilde{\bm{\phi}}$, the element-wise constant gradient $\nabla\widetilde{\phi}_h^K$ is computed as
\begin{equation}
\nabla\widetilde{\phi}_h^K = \sum_{a=1}^{3} \widetilde{\phi}_a \, \nabla\lambda_a^K,
\end{equation}
and the regularised field magnitude on element $K$ is
\begin{equation}
E_{\delta,K} = \sqrt{|\nabla\widetilde{\phi}_h^K|^2 + \delta^2}.
\end{equation}

\subsubsection{Electron Transport Step (Newton for $\mathbf{N}$)}
With the potential fixed at $\widetilde{\bm{\phi}}$, the electron equation is nonlinear in $N$. We solve it by a damped Newton method, starting from the initial guess $\mathbf{N}^{(0)} = \mathbf{N}^k$. For $m = 0, 1, \dots$ (maximum $m_{\mathrm{max}}$ iterations), the following computations are performed.\\

\noindent\textbf{Residual assembly}

On each element $K$, the local residual vector $\mathbf{R}_K^N \in \mathbb{R}^3$ for the current iterate $\mathbf{N}^{(m)}$ has entries (for $a=1,2,3$)
\begin{equation}
(\mathbf{R}_K^N)_a = \sum_{q=1}^{3} w_q \, |K| \,
    \Bigl[ \mathbf{F}_q \cdot \nabla \lambda_a^K
           - (n_{\min} + e^{N_h^{(m)}(\bm{x}_q)}) S_K \, \xi_{q,a} \Bigr]
    + \text{SUPG}_a^K + \text{Bdry}_a^K,
\end{equation}
where
\begin{enumerate}[label=(\arabic*), left=0pt, itemsep=8pt, align=left, widest=2]
    \item $\displaystyle \bm{x}_q = \sum_{c=1}^{3} \xi_{q,c} \bm{x}_c$ is the physical quadrature point;
    \item $\displaystyle N_h^{(m)}(\bm{x}_q) = \sum_{c=1}^{3} N_c^{(m)} \xi_{q,c}$;
    \item $\displaystyle S_K = \epsilon_n E_{\delta,K} 
        \bigl( \alpha_1 e^{-\alpha_2 / E_{\delta,K}} - \eta_0 \bigr)$ is the element-wise constant source coefficient;
    \item the flux vector is
    \[
    \mathbf{F}_q = \epsilon_n e^{N_h^{(m)}(\bm{x}_q)} \nabla N_h^K
        - \bigl( n_{\min} + e^{N_h^{(m)}(\bm{x}_q)} \bigr)
          \bigl( \epsilon_n \nabla \widetilde{\phi}_h^K
                 + \mathbf{v}(\bm{x}_q) \bigr).
    \] with \[\nabla N_h^K = \sum_{c=1}^{3} N_c^{(m)} \nabla \lambda_c^K.\]
\end{enumerate}
The SUPG stabilisation term is
\begin{equation}
\mathrm{SUPG}_a^K = \sum_{q=1}^{3} w_q \, |K| \, 
    \tau_K (\mathbf{w}_n(\bm{x}_q) \cdot \nabla \lambda_a^K)
    \left[ \mathbf{F}_q \cdot \nabla N_h^K
         - (n_{\min} + e^{N_h^{(m)}(\bm{x}_q)}) S_K \right],
\end{equation}
with advective velocity $\mathbf{w}_n(\bm{x}_q) = -(\epsilon_n \nabla \widetilde{\phi}_h^K + \mathbf{v}(\bm{x}_q))$ and element stabilisation parameter $\tau_K$ defined in \eqref{tau_def}.

The boundary contribution $\mathrm{Bdry}_a^K$ is present only if an edge $e \subset \partial K$ lies on $\Gamma_N$. Using a two point Gauss rule on the edge with weights $w_q^{\mathrm{edge}}$ and points $\bm{x}_q^{\mathrm{edge}}$, it is
\begin{equation}
\mathrm{Bdry}_a^K = \sum_{q=1}^{2} w_q^{\mathrm{edge}} \, |e| \, 
    (n_{\min} + e^{N_h^{(m)}(\bm{x}_q^{\mathrm{edge}})})
    (\mathbf{v}(\bm{x}_q^{\mathrm{edge}}) \cdot \nu_e) \, 
    \xi_{q,a}^{\mathrm{edge}},
\end{equation}
where $\xi_{q,a}^{\mathrm{edge}}$ is the value of the local basis function $\lambda_a^K$ at the edge quadrature point. In the present numerical experiments $\Gamma_N = \emptyset$, so this term is omitted.\\

\noindent\textbf{Jacobian assembly}

The local Jacobian matrix $\mathbf{J}_K^N \in \mathbb{R}^{3\times3}$ is obtained by differentiating $\mathbf{R}_K^N$ with respect to the nodal values $N_b$. Its entries are
\begin{equation}
    \begin{split}
     (\mathbf{J}_K^N)_{ab} = \sum_{q=1}^{3} w_q \, |K| \, 
    \left[ \frac{\partial \mathbf{F}_q}{\partial N_b} \cdot \nabla \lambda_a^K
           - e^{N_h^{(m)}(\bm{x}_q)} \xi_{q,b} \, S_K \, \xi_{q,a} \right]\\
    + \frac{\partial}{\partial N_b} \bigl( \mathrm{SUPG}_a^K \bigr)
    + \frac{\partial}{\partial N_b} \bigl( \mathrm{Bdry}_a^K \bigr),
    \end{split}
\end{equation}
with
\begin{equation}
    \begin{split}
    \frac{\partial \mathbf{F}_q}{\partial N_b} = \epsilon_n e^{N_h^{(m)}(\bm{x}_q)} \xi_{q,b} \nabla N_h^K + \epsilon_n e^{N_h^{(m)}(\bm{x}_q)} \nabla \lambda_b^K \\- e^{N_h^{(m)}(\bm{x}_q)} \xi_{q,b} \bigl( \epsilon_n \nabla \widetilde{\phi}_h^K + \mathbf{v}(\bm{x}_q) \bigr)
    \end{split}
\end{equation}
The derivatives of the SUPG and boundary terms are computed analogously by applying the product rule; the full expressions are standard and omitted for brevity.\\

\noindent\textbf{Global assembly and Dirichlet conditions}

The global residual $\mathbf{R}^N$ and Jacobian $\mathbf{J}^N$ are obtained by summing the local contributions into the appropriate nodal indices. For each Dirichlet node $i \in \mathcal{I}_{\mathrm{Dir}}$:
\begin{enumerate}[label=(\arabic*), left=0pt, itemsep=6pt, align=left, widest=2]
    \item $(\mathbf{R}^N)_i \gets N_i^{(m)} - N_D(\bm{x}_i)$;
    \item $(\mathbf{J}^N)_{ii} \gets 1$, and $(\mathbf{J}^N)_{ij} \gets 0$ for all $j \neq i$;
    \item $(\mathbf{J}^N)_{ji} \gets 0$ for all $j \neq i$.
\end{enumerate} 
\noindent\textbf{Newton update and line search} 

The Newton correction $\Delta\mathbf{N} \in \mathbb{R}^{N_n}$ is computed by solving the reduced linear system on the free nodes:
\begin{equation}
\mathbf{J}^N(\mathcal{I}_{\mathrm{free}}, \mathcal{I}_{\mathrm{free}}) \, 
\Delta\mathbf{N}_{\mathrm{free}} = -\mathbf{R}^N(\mathcal{I}_{\mathrm{free}}).
\label{line}
\end{equation}
A backtracking line search determines the step length $\alpha \in (0,1]$ such that the residual norm decreases sufficiently (Armijo condition with parameter $\sigma = 10^{-3}$). The nodal values are then updated as
\begin{equation}
\mathbf{N}^{(m+1)} = \mathbf{N}^{(m)} + \alpha \, \Delta\mathbf{N}.
\end{equation}

To prevent numerical overflow, the components of $\mathbf{N}^{(m+1)}$ are clamped to the interval $[-50, 50]$. The Newton iteration terminates when
\begin{equation}
\frac{\|\Delta\mathbf{N}_{\mathrm{free}}\|}{1 + \|\mathbf{N}_{\mathrm{free}}^{(m)}\|} < \tau_{\mathrm{Newton}},
\end{equation}
with a typical tolerance $\tau_{\mathrm{Newton}} = 10^{-10}$. The converged vector is taken as the intermediate electron log density $\widetilde{\mathbf{N}}$.

\subsubsection{Ion Transport Step (Newton for $\mathbf{P})$}
The ion equation is solved by an entirely analogous Newton procedure, with the frozen fields $\widetilde{\bm{\phi}}$ and $\widetilde{\mathbf{N}}$. Starting from $\mathbf{P}^{(0)} = \mathbf{P}^k$, the local residual and Jacobian are assembled using the ion specific parameters:
\begin{itemize}
    \item Replace $n_{\min}$ with $p_{\min}$ and $\epsilon_n$ with $\epsilon_p$;
    \item The flux vector becomes
    \[
    \mathbf{F}_q^P = \epsilon_p e^{P_h^{(m)}(\bm{x}_q)} \nabla P_h^K
        + (p_{\min} + e^{P_h^{(m)}(\bm{x}_q)})
          \bigl( \epsilon_p \nabla \widetilde{\phi}_h^K - \mathbf{v}(\bm{x}_q) \bigr);
    \]
    \item The advective velocity for SUPG is $\mathbf{w}_p(\bm{x}_q) = \epsilon_p \nabla \widetilde{\phi}_h^K - \mathbf{v}(\bm{x}_q)$;
    \item The source term $S_K$ remains unchanged (it depends only on $\widetilde{\mathbf{N}}$ and $\widetilde{\bm{\phi}}$).
\end{itemize}
The Newton iteration produces the intermediate ion log density $\widetilde{\mathbf{P}}$ after satisfying the same convergence criterion.

After the three sequential steps, the intermediate vectors $\widetilde{\bm{\phi}}$, $\widetilde{\mathbf{P}}$, and $\widetilde{\mathbf{N}}$ have been computed. To enhance stability, an under-relaxation (outer damping) is applied to obtain the new iterates:

\begin{equation}
\begin{cases}
    \phi^{k+1} &= \theta_\phi \widetilde{\phi} + (1-\theta_\phi) \phi^k, \\
    P^{k+1} &= \theta_P \widetilde{P} + (1-\theta_P) P^k, \\
    N^{k+1} &= \theta_N \widetilde{N} + (1-\theta_N) N^k, \label{new}
\end{cases}
\end{equation}

with damping factors $\theta_\phi, \theta_P, \theta_N \in (0,1]$. In practice a common value $\theta \in [0.3,0.7]$ is employed for all three fields. The outer iteration terminates when the maximum relative change in the nodal values satisfies
\begin{equation}
\frac{\|\phi^{k+1} - \phi^k\|}{1 + \|\phi^k\|}
+ \frac{\|P^{k+1} - P^k\|}{1 + \|P^k\|}
+ \frac{\|N^{k+1} - N^k\|}{1 + \|N^k\|}
< \tau_{\mathrm{outer}},
\end{equation}
with a typical tolerance $\tau_{\mathrm{outer}} = 10^{-8}$.

The nonlinear algorithm described above can be summarized as Algorithm \ref{alg:cap}.
\begin{algorithm}[H]
\caption{Gummel‑Newton iteration for the shifted‑log formulation}\label{alg:cap}
\begin{algorithmic}
\Require Mesh, physical parameters, Dirichlet data $\phi_D, P_D, N_D$, initial guesses $\phi^0, P^0, N^0$, damping factors $\theta_{\phi}, \theta_P, \theta_N \in (0, 1]$, tolerances $\tau_{outer}, \tau_{Newton}$.
\State \textbf{Initialisation}: $k=0$\\

\State \textbf{Step 1}: 
Assemble $A_{\phi}$ and $f_{\phi}$ with $P^k$ and $N^k$  frozen. Solve (\ref{poisson}) with Dirichlet conditions.\\
$\quad \rightarrow$ intermediate potential $\tilde{\phi}$.\\

\State \textbf{Step 2}: Assemble residual $R^N$ and Jacobian $J^N$. Solve (\ref{line}); update via backtracking line search.\\
$\quad \rightarrow$ intermediate electron log‑density $\tilde{N}$.\\

\State \textbf{Step 3}: Use analogous Newton procedure  with SUPG‑stabilised to solve the P-equation, with $\tilde{\phi}$ and $\tilde{N}$
frozen.\\
$\quad \rightarrow$ intermediate ion log‑density
$\tilde{P}$.\\

\State \textbf{Step 4}: Compute the new iterates (\ref{new}).\\

\State If the total relative change in the three fields is below $\tau_{outer}$, return 
$(\phi^{k+1}, P^{k+1}, N^{k+1})$ as the solution.

\State Else set $k \gets k+1$ and go to \textbf{Step 1}.
\end{algorithmic}
\end{algorithm}
~\\


			\bibliographystyle{abbrv}
			\bibliography{Bibliography}

@article{budd1991coronas,
	title={Coronas and the space charge problem},
	author={Budd, Chris},
	journal={European Journal of Applied Mathematics},
	volume={2},
	number={1},
	pages={43--81},
	year={1991},
	publisher={Cambridge University Press}
}

@book{fichera1965linear,
	title={Linear elliptic differential systems and eigenvalue problems},
	author={Fichera, Gaetano},
	volume={8},
	year={1965},
	publisher={Springer}
}

@article{ladyzhenskaya1968linear,
	title={Linear and quasilinear elliptic equations, 1968},
	author={Ladyzhenskaya, Olga A and Ural’tseva, Nina N},
	journal={Leon Ehrenpreis Academic Press, New York},
	year={1968}
}

@book{krasnoselskij1984geometrical,
	title={Geometrical methods of nonlinear analysis},
	author={Krasnoselskij, Mark Aleksandrovi{\v{c}} and Zabrejko, Petr Petrovi{\v{c}}},
	volume={263},
	year={1984},
	publisher={Springer}
}

@article{rabinowitz1971some,
	title={Some global results for nonlinear eigenvalue problems},
	author={Rabinowitz, Paul H},
	journal={Journal of functional analysis},
	volume={7},
	number={3},
	pages={487--513},
	year={1971},
	publisher={Elsevier}
}

@book{markowich1985stationary,
	title={The stationary semiconductor device equations},
	author={Markowich, Peter A},
	year={1985},
	publisher={Springer Science \& Business Media}
}

@book{markowich2012semiconductor,
	title={Semiconductor equations},
	author={Markowich, Peter A and Ringhofer, Christian A and Schmeiser, Christian},
	year={2012},
	publisher={Springer Science \& Business Media}
}

@article{frehse1993existence,
	author = {J. Frehse and J. Naumann},
	title = {An existence theorem for weak solutions of the basic stationary semiconductor equations},
	journal = {Applicable Analysis},
	volume = {48},
	number = {1-4},
	pages = {157--172},
	year = {1993},
	publisher = {Taylor \& Francis},
	doi = {10.1080/00036819308840156},
	URL = { https://doi.org/10.1080/00036819308840156},
	eprint = { 	https://doi.org/10.1080/00036819308840156}}

@inproceedings{stampacchia1965probleme,
	title={Le probl{\`e}me de Dirichlet pour les {\'e}quations elliptiques du second ordre {\`a} coefficients discontinus},
	author={Stampacchia, Guido},
	booktitle={Annales de l'institut Fourier},
	volume={15},
	pages={189--257},
	year={1965}
}

@book{renardy2004introduction,
  title={An introduction to partial differential equations},
  author={Renardy, Michael and Rogers, Robert C},
  year={2004},
  publisher={Springer}
}

@article{lions1969quelques,
  title={" Quelques M{\'e}thodes de R{\'e}solution des Probl{\`e}mes aux Limites Non-Lin{\'e}aires,''},
  author={Lions, J-L},
  journal={Dunod},
  year={1969}
}

@book{evans2022partial,
  title={Partial differential equations},
  author={Evans, Lawrence C},
  volume={19},
  year={2022},
  publisher={American mathematical society}
}

@article{Budd1991,
  author  = {Budd, C.},
  title   = {Coronas and the space charge problem},
  journal = {European Journal of Applied Mathematics},
  year    = {1991},
  volume  = {2},
  pages   = {43--81},
}

@book{Grisvard1992,
  author    = {Grisvard, P.},
  title     = {Singularities in Boundary Value Problems},
  publisher = {Masson},
  year      = {1992},
  address   = {Paris},
}

@article{Molinero2006,
  author  = {Molinero, D. and Comulada, R. and Castañer, L.},
  title   = {Dielectric charge measurements in capacitive microelectromechanical switches},
  journal = {Applied Physics Letters},
  year    = {2006},
  volume  = {89},
  pages   = {103506},
  doi     = {10.1063/1.2345609},
}

@article{Morrow1997,
  author  = {Morrow, R. and Lowke, J. J.},
  title   = {Streamer propagation in air},
  journal = {Journal of Physics D: Applied Physics},
  year    = {1997},
  volume  = {30},
  pages   = {614--627},
  doi     = {10.1088/0022-3727/30/4/017},
}

@book{Raizer1991,
  author    = {Raizer, Y. P.},
  title     = {Gas Discharge Physics},
  publisher = {Springer},
  year      = {1991},
  address   = {Berlin},
  isbn      = {978-3-642-64760-4},
}

@book{Roache1998,
  author    = {Roache, P. J.},
  title     = {Verification and Validation in Computational Science and Engineering},
  publisher = {Hermosa},
  year      = {1998},
  address   = {Albuquerque},
  isbn      = {978-0-913478-08-2},
}

@book{Senturia2001,
  author    = {Senturia, S. D.},
  title     = {Microsystem Design},
  publisher = {Kluwer},
  year      = {2001},
  address   = {Boston},
  isbn      = {978-0-7923-7246-2},
}

@book{Zienkiewicz2013,
  author    = {Zienkiewicz, O. C. and Taylor, R. L. and Zhu, J. Z.},
  title     = {The Finite Element Method: Its Basis and Fundamentals},
  edition   = {7th},
  publisher = {Butterworth-Heinemann},
  year      = {2013},
  address   = {Oxford},
  isbn      = {978-1-85617-633-0},
}

@book{rebeiz2004rf,
  title={RF MEMS: theory, design, and technology},
  author={Rebeiz, Gabriel M},
  year={2004},
  publisher={John Wiley \& Sons}
}

@article{van2012capacitive,
  title={Capacitive RF MEMS switch dielectric charging and reliability: a critical review with recommendations},
  author={Van Spengen, WM},
  journal={Journal of Micromechanics and Microengineering},
  volume={22},
  number={7},
  pages={074001},
  year={2012},
  publisher={IOP Publishing}
}

@article{papaioannou2005temperature,
  title={Temperature study of the dielectric polarization effects of capacitive RF MEMS switches},
  author={Papaioannou, Giorgos and Exarchos, M-N and Theonas, Vasilios and Wang, Guoan and Papapolymerou, John},
  journal={IEEE Transactions on Microwave Theory and Techniques},
  volume={53},
  number={11},
  pages={3467--3473},
  year={2005},
  publisher={IEEE}
}

@inproceedings{yuan2005modeling,
  title={Modeling and characterization of dielectric-charging effects in RF MEMS capacitive switches},
  author={Yuan, Xiaobin and Hwang, James CM and Forehand, David and Goldsmith, Charles L},
  booktitle={IEEE MTT-S International Microwave Symposium Digest, 2005.},
  pages={753--756},
  year={2005},
  organization={IEEE}
}

@inproceedings{herfst2008kelvin,
  title={Kelvin probe study of laterally inhomogeneous dielectric charging and charge diffusion in RF MEMS capacitive switches},
  author={Herfst, RW and Steeneken, PG and Schmitz, J and Mank, AJG and Van Gils, M},
  booktitle={2008 IEEE International Reliability Physics Symposium},
  pages={492--495},
  year={2008},
  organization={IEEE}
}

@article{molinero2006dielectric,
  title={Dielectric charge measurements in capacitive microelectromechanical switches},
  author={Molinero, D and Comulada, R and Castaner, L},
  journal={Applied physics letters},
  volume={89},
  number={10},
  year={2006},
  publisher={AIP Publishing}
}

@article{van1950theory,
  title={Theory of the flow of electrons and holes in germanium and other semiconductors},
  author={Van Roosbroeck, W},
  journal={The Bell System Technical Journal},
  volume={29},
  number={4},
  pages={560--607},
  year={1950},
  publisher={Nokia Bell Labs}
}

@book{selberherr1984analysis,
  title={Analysis and simulation of semiconductor devices},
  author={Selberherr, Siegfried},
  year={1984},
  publisher={Springer Science \& Business Media}
}

@book{jungel2001quasi,
  title={Quasi-hydrodynamic semiconductor equations},
  author={J{\"u}ngel, Ansgar},
  volume={41},
  year={2001},
  publisher={Springer Science \& Business Media}
}

@article{brezzi2005discretization,
  title={Discretization of semiconductor device problems (I)},
  author={Brezzi, Franco and Marini, LUISA DONATELLA and Micheletti, Stefano and Pietra, Paola and Sacco, Riccardo and Wang, Song},
  journal={Handbook of numerical analysis},
  volume={13},
  pages={317--441},
  year={2005},
  publisher={Elsevier}
}

@article{miller1999application,
  title={Application of finite element methods to the simulation of semiconductor devices},
  author={Miller, John James Henry and Schilders, WHA and Wang, Song},
  journal={Reports on Progress in Physics},
  volume={62},
  number={3},
  pages={277--353},
  year={1999}
}

@book{vasileska2017computational,
  title={Computational Electronics: semiclassical and quantum device modeling and simulation},
  author={Vasileska, Dragica and Goodnick, Stephen M and Klimeck, Gerhard},
  year={2017},
  publisher={CRC press}
}

@book{townsend1915electricity,
  title={Electricity in gases},
  author={Townsend, John Sealy},
  year={1915},
  publisher={Clarendon Press, Oxford }
}

@book{xiao2016gas,
  title={Gas discharge and gas insulation},
  author={Xiao, Dengming},
  year={2016},
  publisher={Springer}
}

@book{townsend1910theory,
  title={The theory of ionization of gases by collision},
  author={Townsend, John},
  year={1910},
  publisher={Constable, Limited}
}

@article{morrow1997streamer,
  title={Streamer propagation in air},
  author={Morrow, Richard and Lowke, John J},
  journal={Journal of Physics D: Applied Physics},
  volume={30},
  number={4},
  pages={614--627},
  year={1997}
}

@book{roos2008robust,
  title={Robust numerical methods for singularly perturbed differential equations: convection-diffusion-reaction and flow problems},
  author={Roos, Hans-G{\"o}rg and Stynes, Martin and Tobiska, Lutz},
  year={2008},
  publisher={Springer}
}

@article{brooks1982streamline,
  title={Streamline upwind/Petrov-Galerkin formulations for convection dominated flows with particular emphasis on the incompressible Navier-Stokes equations},
  author={Brooks, Alexander N and Hughes, Thomas JR},
  journal={Computer methods in applied mechanics and engineering},
  volume={32},
  number={1-3},
  pages={199--259},
  year={1982},
  publisher={Elsevier}
}

@article{hughes2018multiscale,
  title={Multiscale and stabilized methods},
  author={Hughes, Thomas JR and Scovazzi, Guglielmo and Franca, Leopoldo P},
  journal={Encyclopedia of computational mechanics second edition},
  pages={1--64},
  year={2018},
  publisher={Wiley Online Library}
}

@book{patankar2018numerical,
  title={Numerical heat transfer and fluid flow},
  author={Patankar, Suhas},
  year={2018},
  publisher={CRC press}
}

@article{chainais2003finite,
  title={Finite volume scheme for multi-dimensional drift-diffusion equations and convergence analysis},
  author={Chainais-Hillairet, Claire and Liu, Jian-Guo and Peng, Yue-Jun},
  journal={ESAIM: Mod{\'e}lisation math{\'e}matique et analyse num{\'e}rique},
  volume={37},
  number={2},
  pages={319--338},
  year={2003}
}

@article{abdel2025existence,
  title={Existence of solutions and uniform bounds for the stationary semiconductor equations with generation and ionic carriers},
  author={Abdel, Dilara and Blaustein, Alain and Chainais-Hillairet, Claire and Herda, Maxime and Moatti, Julien},
  journal={arXiv preprint arXiv:2511.23250},
  year={2025}
}

@article{gajewski1986basic,
  title={On the basic equations for carrier transport in semiconductors},
  author={Gajewski, Herbert and Gr{\"o}ger, Konrad},
  journal={Journal of mathematical analysis and applications},
  volume={113},
  number={1},
  pages={12--35},
  year={1986},
  publisher={Academic Press}
}

@incollection{degond2004note,
  title={A note on the energy-transport limit of the semiconductor Boltzmann equation},
  author={Degond, Pierre and Levermore, C David and Schmeiser, Christian},
  booktitle={Transport in Transition Regimes},
  pages={137--153},
  year={2004},
  publisher={Springer}
}

@article{degond2000numerical,
  title={Numerical discretization of energy-transport models for semiconductors with nonparabolic band structure},
  author={Degond, Pierre and J{\"u}ngel, Ansgar and Pietra, Paola},
  journal={SIAM Journal on Scientific Computing},
  volume={22},
  number={3},
  pages={986--1007},
  year={2000},
  publisher={SIAM}
}

@article{scharfetter2005large,
  title={Large-signal analysis of a silicon read diode oscillator},
  author={Scharfetter, Donald L and Gummel, Hermann K},
  journal={IEEE Transactions on electron devices},
  volume={16},
  number={1},
  pages={64--77},
  year={2005},
  publisher={IEEE}
}

@inproceedings{kumar2017fem,
  title={An FEM based framework to simulate semiconductor devices using streamline upwind Petrov-Galerkin stabilization technique},
  author={Kumar, Gaurav and Singh, Mandeep and Ray, Ashok and Trivedi, Gaurav},
  booktitle={2017 27th International Conference Radioelektronika (RADIOELEKTRONIKA)},
  pages={1--5},
  year={2017},
  organization={IEEE}
}

@article{chen2020steady,
  title={Steady-state simulation of semiconductor devices using discontinuous Galerkin methods},
  author={Chen, Liang and Bagci, Hakan},
  journal={IEEE Access},
  volume={8},
  pages={16203--16215},
  year={2020},
  publisher={IEEE}
}

@article{liu2016analysis,
  title={Analysis of the local discontinuous Galerkin method for the drift-diffusion model of semiconductor devices},
  author={Liu, YunXian and Shu, Chi-Wang},
  journal={Science China Mathematics},
  volume={59},
  number={1},
  pages={115--140},
  year={2016},
  publisher={Springer}
}

@article{jungel2001positivity,
  title={A positivity-preserving numerical scheme for a nonlinear fourth order parabolic system},
  author={J{\"u}ngel, Ansgar},
  journal={SIAM journal on numerical analysis},
  volume={39},
  number={2},
  pages={385--406},
  year={2001},
  publisher={SIAM}
}

@article{brezzi1989two,
  title={Two-dimensional exponential fitting and applications to drift-diffusion models},
  author={Brezzi, Franco and Marini, Luisa Donatella and Pietra, Paola},
  journal={SIAM Journal on Numerical Analysis},
  volume={26},
  number={6},
  pages={1342--1355},
  year={1989},
  publisher={SIAM}
}

@article{polak1987semiconductor,
  title={Semiconductor device modelling from the numerical point of view},
  author={Polak, SJ and Den Heijer, C and Schilders, WHA and Markowich, Peter},
  journal={International Journal for Numerical Methods in Engineering},
  volume={24},
  number={4},
  pages={763--838},
  year={1987},
  publisher={Wiley Online Library}
}

@article{perez2025comparison,
  title={A comparison of formulations and non-linear solvers for computational modelling of semiconductor devices},
  author={P{\'e}rez-Escudero, Sergi and Codony, David and Arias, Irene and Fern{\'a}ndez-M{\'e}ndez, Sonia},
  journal={Computational mechanics},
  volume={75},
  number={5},
  pages={1533--1554},
  year={2025},
  publisher={Springer}
}

@article{gummel2005self,
  title={A self-consistent iterative scheme for one-dimensional steady state transistor calculations},
  author={Gummel, Hermann K},
  journal={IEEE Transactions on electron devices},
  volume={11},
  number={10},
  pages={455--465},
  year={2005},
  publisher={IEEE}
}

@book{gilbarg1998elliptic,
  author    = {David Gilbarg and Neil S. Trudinger},
  title     = {Elliptic Partial Differential Equations of Second Order},
  edition   = {Classics in Mathematics},
  publisher = {Springer},
  address   = {Berlin},
  year      = {2001}
}

@article{frehse1996stationary,
  title={Stationary semiconductor equations modeling avalanche generation},
  author={Frehse, J and Naumann, J},
  journal={Journal of mathematical analysis and applications},
  volume={198},
  number={3},
  pages={685--702},
  year={1996},
  publisher={Elsevier}
}

@article{frehse1994existence,
  title={On the existence of weak solutions to a system of stationary semiconductor equations with avalanche generation},
  author={Frehse, Jens and Naumann, Joachim},
  journal={Mathematical Models and Methods in Applied Sciences},
  volume={4},
  number={02},
  pages={273--289},
  year={1994},
  publisher={World Scientific}
}

@article{gimperlein2026analysis,
  title={Analysis of a Model for Electrical Discharge in MEMS},
  author={Gimperlein, Heiko and He, Runan and Lacey, Andrew A},
  journal={arXiv preprint arXiv:2602.21439},
  year={2026}
}

@article{islamov2003global,
  title={Global Solutions of the Drift-Diffusion Approximation Equations in Gas Discharge Theory.},
  author={Islamov, R Sh},
  journal={Differential Equations},
  volume={39},
  number={12},
  year={2003}
}

@article{islamov2006regularity,
  title={On the regularity of solutions to the drift-diffusion approximation equations in gas discharge theory},
  author={Islamov, R Sh},
  journal={Computational Mathematics and Mathematical Physics},
  volume={46},
  number={1},
  pages={125--140},
  year={2006},
  publisher={Springer}
}
		
	\end{document}